\documentclass[10pt,twoside,a4paper]{article}
\usepackage[table,xcdraw]{xcolor}
\usepackage[utf8]{inputenc}
\usepackage[T1]{fontenc}
\usepackage{textcomp}
\usepackage[english]{babel}
\usepackage{amsmath,amssymb,amsfonts,mathrsfs,amsthm}
\usepackage{url}
\usepackage{hyperref}
\hypersetup{colorlinks=true,linkcolor = blue,
            urlcolor  = blue,
            citecolor = blue,
            anchorcolor = blue}
\usepackage{enumitem}
\setlist[description]{font=\bfseries}
\usepackage{parskip}
\usepackage{emptypage}
\usepackage{subcaption}
\usepackage{multicol}
\usepackage{dsfont}
\usepackage{thmtools}
\usepackage{authblk}

\usepackage{mathrsfs}
\usepackage[margin=3cm]{geometry}
\usepackage{csquotes}
\usepackage{lmodern}
\usepackage{pdfsync}

\usepackage{fancyhdr}
\usepackage{hyphenat}

\usepackage{longtable}
\usepackage{ifthen} 
\usepackage{newunicodechar}
\usepackage{tikz-cd}%to build commutative diagrams
\definecolor{Myblue}{rgb}{0.2,0.2,0.7}
\usetikzlibrary{arrows}
\usepackage{tabularx}
\usepackage{mathtools}
\usepackage{comment}
\usepackage{bm} % For bold math symbols
\usepackage[capitalize]{cleveref}
\usepackage{indentfirst}

\usepackage[doi=false,isbn=false,url=false,eprint=true,backend=biber,style=alphabetic,maxnames=5]{biblatex}

\usepackage[normalem]{ulem}

\numberwithin{equation}{section}

\usepackage{array}
\newcolumntype{P}[1]{>{\raggedright\arraybackslash}p{#1}}
\usepackage{xparse}

\usepackage[final]{showlabels}
\usepackage{adjustbox}
\usepackage{xspace}
\usepackage{IEEEtrantools}
\newenvironment{Rem}{\textbf{Remark :}}{}

\declaretheorem[name=Theorem, numberwithin=section,preheadhook=\vspace{1em}]{The}
\declaretheorem[name=Lemma, sibling=The,preheadhook=\vspace{1em}]{Lem}
\declaretheorem[name=Proposition, sibling=The,preheadhook=\vspace{1em}]{Prop}
\declaretheorem[name=Definition, sibling=The,preheadhook=\vspace{1em}]{Def}
\declaretheorem[name=Example, sibling=The,preheadhook=\vspace{1em}]{Ex}
\declaretheorem[name=Corollary, sibling=The,preheadhook=\vspace{1em}]{Cor}
\declaretheorem[name=Theorem, numberwithin=section,preheadhook=\vspace{1em}]{mainthm}
\declaretheorem[name=Corollary,sibling=mainthm,preheadhook=\vspace{1em}]{maincor}

\declaretheorem[name=Proposition, unnumbered, preheadhook=\vspace{1em}]{Propnn}
\declaretheorem[name=Application,unnumbered,preheadhook=\vspace{1em}]{applnn}

\crefformat{The}{#2Theorem~#1#3}
\Crefformat{The}{#2Theorem~#1#3}
\crefformat{Cor}{#2Corollary~#1#3}
\Crefformat{Cor}{#2Corollary~#1#3}
\crefformat{Prop}{#2Proposition~#1#3}
\Crefformat{Prop}{#2Proposition~#1#3}
\crefformat{Lem}{#2Lemma~#1#3}
\Crefformat{Lem}{#2Lemma~#1#3}

\crefformat{mainthm}{#2Theorem~#1#3}
\Crefformat{mainthm}{#2Theorem~#1#3}
\crefformat{maincor}{#2Corollary~#1#3}
\Crefformat{maincor}{#2Corollary~#1#3}
\Crefformat{applnn}{#2Application~#1#3}
\crefformat{applnn}{#2Application~#1#3}

\newtheoremstyle{withtitle}% <nom>
  {3pt}% <espace avant>
  {3pt}% <espace après>
  {\itshape}% <corps de la police>
  {}% <indentation>
  {\bfseries}% <corps du théorème>
  {.}% <ponctuation après le théorème>
  {.5em}% <espace après la ponctuation>
  {\thmname{#1}\thmnumber{ #2}\thmnote{ (#3)}}% <format du titre>
  \theoremstyle{withtitle}
\renewcommand{\phi}{\varphi}
\renewcommand{\epsilon}{\varepsilon}

\newcommand{\zz}{\mathbb Z}

\newcommand{\kk}{\mathbb K}
\newcommand{\nn}{\mathbb N}
\newcommand{\rr}{\mathbb R}

\newcommand{\cc}{\mathbb C}

\newcommand{\ff}{\mathbb F}

\newcommand{\B}{\mathcal B}
\newcommand{\C}{\mathcal C}
\newcommand{\E}{\mathcal E}

\newcommand{\G}{\mathcal G}
\renewcommand{\H}{\mathcal H}
\newcommand{\K}{\mathcal K}
\newcommand{\M}{M}

\newcommand{\U}{\mathcal U}
\newcommand{\Z}{\mathcal Z}
\newcommand{\N}{\mathcal N}
\newcommand{\T}{\mathcal T}
\newcommand{\V}{\mathcal V}
\renewcommand{\O}{\mathcal O}

\newcommand{\Sym}{\mathfrak S}
\newcommand{\Lp}{L}

\newcommand{\pt}{\otimes}
\newcommand{\ptvn}{\overline{\otimes}}

\DeclareMathOperator{\id}{id}

\DeclareMathOperator{\Sp}{Sp}

\DeclareMathOperator{\1}{1}

\DeclareMathOperator{\Aut}{Aut}

\DeclareMathOperator{\I}{I}
\DeclareMathOperator{\II}{II}
\DeclareMathOperator{\III}{III}

\DeclareMathOperator{\vect}{span}

\newcommand{\Linf}{\Lp^{\infty}}
\newcommand{\Ldeux}{\Lp^{2}}
\newcommand{\ldeux}{\ell^{2}}
\newcommand{\Lun}{\Lp^{1}}

\newcommand{\abs}[1]{\left \lvert #1 \right \rvert}
\newcommand{\norm}[1]{\left \Vert #1 \right \Vert}
\newcommand{\ps}[2]{\left \langle #1,#2 \right \rangle}

\newcommand{\tend}[3][]{\displaystyle\mathop{\overset{#1}{\longrightarrow}}_{#2\rightarrow#3}} 

\newcommand{\acts}{\curvearrowright}

\newcommand{\ie}{i.\,e.\xspace}

\newcommand{\Cstar}{\operatorname{C}^*}

\renewcommand{\L}{\mathcal L}
\newcommand{\wot}{\operatorname{wot}}
\newcommand{\sot}{\operatorname{sot}}
\newcommand{\op}{\operatorname{op}}
\newcommand{\Wstar}{\operatorname{W}^*}
\let\Sp\relax
\DeclareMathOperator{\Sp}{Sp}
\DeclareMathOperator{\Sinv}{S}
\DeclareMathOperator{\Tinv}{T}

\newcommand{\cl}{\leq}

\newcommand{\co}{<_{co}}
\newcommand{\ind}[2]{\left[#1:#2\right]}

\newcommand{\nmlz}[2]{\N_{#1}(#2)}
\newcommand{\Ccc}[2]{\C_{#1}(#2)}%\newcommand{\conj}[2][]{_{#1}#2}
\newcommand{\ccc}[2]{c_{#1}(#2)}

\newcommand{\fc}[2]{%
  \widehat{#1}%
  \if\relax\detokenize{#2}\relax
  \else
    (#2)% Sinon → on met les parenthèses
  \fi
}
\newcommand{\etoile}[1]{\hyperref[eq:etoileH]{(\hspace*{1mm}\star_{#1})}}
\newcommand{\lcocosets}[1]{\operatorname{Cosets_{co}}(#1)}

\newcommand{\cosub}[1]{\operatorname{Sub_{co}}(#1)}

\DeclareMathOperator{\Homeo}{Homeo}
\DeclareMathOperator{\Fix}{Fix}
\DeclareMathOperator{\Stab}{Stab}
\DeclareMathOperator{\Isom}{Isom}

\DeclareMathOperator{\rist}{rist}

\newcommand{\fp}[2]{#1^{#2}}

\newcommand{\Tdk}{\T_{d,k}}
\newcommand{\Tdkn}{\T_{d,k}^{\leq n}}
\newcommand{\Tdd}{\T_{d,d}}

\newcommand{\bTdk}{\partial\T_{d,k}}

\newcommand{\Ndk}{\N_{d,k}}
\newcommand{\Odk}{\O_{d,k}}
\newcommand{\Kdk}{K_{d,k}}
\newcommand{\Odkn}[1]{\O_{d,k}^{(#1)}}
\newcommand{\Kdkn}[1]{K_{d,k}^{(#1)}}

\DeclareMathOperator{\HNN}{HNN}

\renewcommand{\bm}[1]{U(#1)}
\newcommand{\bmp}[1]{\widetilde{U}(#1)}

\newcommand{\properideal}{%
  \mathrel{\ooalign{$\lneq$\cr\raise.22ex\hbox{$\lhd$}\cr}}}

\begin{document}

\title{The regular representation of Neretin groups is factorial}
    \author{Basile Morando\thanks{UMPA, ENS Lyon, France}}
    \date{\today}
    \maketitle
\begin{abstract}
 We show that the left regular representation of Neretin groups is factorial, providing the first example of a non-discrete simple group with this property. This is based on a new criterion of factoriality for totally disconnected groups. For groups $G$ satisfying the criterion, we determine the type of the factor $\L(G)$ and derive factoriality results for crossed products associated to $G$-actions on von Neumann algebras.
\end{abstract}

\section*{Introduction}
\addcontentsline{toc}{section}{Introduction}

Let $(\pi, \H)$ be a unitary representation of a locally compact group $G$. It is \emph{irreducible} if it cannot be decomposed as a direct sum of two subrepresentations, and \emph{factorial} if it cannot be decomposed as a direct sum of two \emph{disjoint} subrepresentations. These properties are encoded in the commutant of the representation $\pi(G)'=\{T\in \B(\H)\mid \forall g\in G, T\pi(g)=\pi(g)T\}$. The representation $\pi$ is irreducible if and only if $\pi(G)'$ is trivial, and $\pi$ is factorial if and only if $\pi(G)'$ has a trivial center.

Any multiple of an irreducible representation is factorial, but the converse does not hold in general. A second-countable locally compact group $G$ is said to be of \emph{type }$\I$ if every factorial unitary representation of $G$ is a multiple of an irreducible representation. The class
of type $\I$ groups includes, among others, compact groups, abelian locally compact groups, connected semisimple Lie groups, connected nilpotent Lie groups and the group of automorphisms of a regular tree (see \cite[Theorem 6.E.19]{BekkadekaHarpe2020UnitaryRepresentations}). 

The terms \emph{factorial representation} and \emph{type $\I$ group} are adopted from the theory of von Neumann algebras introduced by Murray and von Neumann in their seminal papers \cite{MurrayvonNeumann1936RingsOperators}. Von Neumann algebras can be defined as algebras of bounded operators on a Hilbert space that are equal to their double commutant. Consequently, $\pi(G)'$ and $\pi(G)''$ are natural examples of von Neumann algebras. A von Neumann algebra is a \emph{factor} if its center is trivial, and such factors have been classified by Murray and von Neumann into types $\I$, $\II$ and $\III$ based on the geometry of their lattices of projections. Type $\I$ factors are exactly the factors of the form $\B(\H_1)\otimes \1_{\H_2}$ acting on the Hilbert space $\H_1\otimes \H_2$. A representation $\pi$ is factorial if and only if the von Neumann algebra $\pi(G)'$ (equivalently $\pi(G)''$) is a factor, and a second-countable group is of type $\I$ if and only if $\pi(G)''$ is of type $\I$ for every factorial representation $\pi$ of $G$.

In the same series of papers, Murray and von Neumann highlight that not every group is of type $\I$ by examining left and right regular representations $\lambda_{G}$ and $\rho_G$ of discrete groups. They investigate the associated \emph{group von Neumann algebra} 
\[ \lambda_G(G)'=\rho_G(G)''\simeq \lambda_G(G)''\doteq \L(G).\] 
They show that when $G$ is discrete, $\lambda_{G}$ is factorial if and only if the nontrivial conjugacy classes of $G$ are infinite: the resulting von Neumann algebra $\L(G)$ is then a type $\II$ factor, as it admits a faithful normal tracial state. 

While this result completely settles the question of the factoriality of the regular representation of discrete groups, obtaining a similar result for non-discrete locally compact groups is a notoriously tricky problem (see, for instance, \cite{Vaes2025factorialitytwistedlocallycompact} which highlights fundamentally new phenomena in the non-discrete case). Over the years, several examples of non-discrete groups whose group von Neumann algebra is a factor appeared in the literature.

\begin{itemize}

\item First examples are based on the following proposition, which is a direct consequence of \cite[Theorem VIII]{vonNeumann1940RingsOperatorsIII}, \cite[Corollaire 2.2]{Sauvageot1977TypeProduitCroise} and \cite[Proposition 2.2]{Sutherland1978TypeAnalysis}.
\begin{Propnn}[{\textmd{\cite[Theorem VIII]{vonNeumann1940RingsOperatorsIII}}}]\label{Prop:factoriality of group measure space construction}
    Let $\alpha:G\to \Aut(N)$ be a continuous action by automorphisms of a locally compact group $G$ on an abelian locally compact group $N$. Denote by $\hat{\alpha}:G\acts \hat{N}$ the associated dual action on the Pontryagin dual of $N$. If $\alpha$ is essentially free and ergodic, then $\L(\hat{N}\rtimes_{\hat{\alpha}}G)$ is a factor.
\end{Propnn}

For example, given a local field $\kk$,  $\L(\kk\rtimes \kk^*)$ is a factor of type $\I$. Semi-restricted direct products of such groups are used in \cite{Blackadar1977RegularRepresentation} (and more recently in \cite{Vaes2025factorialitytwistedlocallycompact}) to give examples of non-discrete groups such that $\L(G)$ is a factor of type $\III$. Sutherland applies the same theorem in \cite{Sutherland1978TypeAnalysis} to discrete matrix groups acting on $\rr^2$ and exhibits groups $G$ with $\L(G)$ being a factor of any type.

\item A second source of examples is the literature related to the companion question of establishing $\Cstar$-simplicity for some non-discrete locally compact groups\footnote{We say that a group $G$ is $\Cstar$-simple if its reduced $\Cstar$-algebra $\Cstar_r(G)$ is simple. As $\L(G)$ is the \emph{weak} completion of $\Cstar_r(G)$, we are basically investigating ``$\Wstar$-simplicity'' for locally compact groups, von Neumann algebras being also called $\Wstar$-algebras.}.  The main theorem of \cite{Suzuki2017Elementary} establishes $\Cstar$-simplicity for a class of totally disconnected locally compact groups ``well approximated'' by $\Cstar$-simple discrete groups. Following the same lines of proof, showing that some locally compact groups that are ``well approximated'' by discrete groups with infinite conjugacy classes have factorial regular representation is straightforward. Using this criterion of $\Cstar$-simplicity, Raum shows in \cite{Raum2021Erratum} that the non-discrete profinite completions $G(m,n)$ of  Baumslag-Solitar groups are $\Cstar$-simple (see \cite{Mukohara2024Csimplicity} for generalizations), and the same proof implies that their regular representations are factorial (see \Cref{subsection:HNN extensions} for details). Still motivated by $\Cstar$-simplicity and given any locally compact group $H$, Suzuki builds in \cite[Theorem 5.1]{Suzuki2022Csimplicity} a group $G$ having $H$ as an open subgroup and such that $\L(G)$ is a factor: this shows that neither the $\Cstar$-simplicity of a group nor the factoriality of its von Neumann algebra admit local obstructions.

\item In the article \cite{BeltitaBelti2024RegularRepresentation}, the authors give a criterion of factoriality for the regular representation of solvable connected and simply connected Lie groups in terms of the orbit of the coadjoint representation of the group, following a rich literature on representation theory of solvable Lie groups developed among others by Kirillov a Pukánsky.

\item Recently, the preprints  \cite{GuintoNelson2025AlmostUnimodularGroups} and \cite[Theorem 4.5]{Miyamoto2025PlancherelWeight} explain how the factoriality of the von Neumann algebra of a non-unimodular locally compact group $G$ relates to the center of the von Neumann algebra of its unimodular part $\L(\ker \Delta_G)$.
\end{itemize}

In this article, we present a new criterion of factoriality for totally disconnected second countable locally compact groups von Neumann algebras, that applies most significantly to the family of Neretin groups introduced in \cite{Neretin1992Combinatorial}. Given $k,d\geq 2$, the Neretin group $\Ndk$ can be defined equivalently as the group of almost automorphisms of a rooted tree $\Tdk$ or as the group of local similarities of the boundary of the same tree (see \cite{GarncarekLazarovich2018NeretinGroups} for a detailed account on the Neretin groups). These groups have attracted much attention over the years due to their remarkable properties: they are abstractly simple \cite{Kapoudjian1999Simplicity}, compactly presented \cite{LeBoudec2017CompactPresentability}, are the first examples of locally compact groups with no lattices \cite{BaderCapraceGelanderMozes2012SimpleGroups} and admit no nontrivial invariant random subgroups \cite{Zheng2019NeretinGroups}. The question of the $\Cstar$-simplicity of the Neretin group is an interesting and still open problem. 

Any Neretin group $\Ndk$ contains an open amenable subgroup consisting of local isometries of the boundary of the tree, denoted by $\Odk$. Note that $\L(\Ndk)$ contains $\L(\Odk)$ as a von Neumann subalgebra: more generally, if $G$ is a locally compact group and $H$ is a closed subgroup of $G$, then 
\[\L(H)=\lambda_{H}(H)''\simeq \lambda_G(H)''\subset \lambda_G(G)''=\L(G).\]
We now state our main theorem: 
\begin{mainthm}[\Cref{thm:Neretin}]\label{mainthm:Neretin} The regular representation $\lambda_{\Ndk}$ of a Neretin group $\Ndk$ is factorial. More precisely, denoting by $\Odk$ is its open amenable subgroup of local isometries, we have
    \[\L(\Odk)'\cap\L(\Ndk)=\cc.\]
As a consequence, $\L(\Odk)$ is the hyperfinite $\II_{\infty}$ factor and $\L(\Ndk)$ is a type $\II_{\infty}$ factor.
\end{mainthm}
To our knowledge, this provides the first example of a non-discrete simple group whose von Neumann algebra is a factor. The type of $\L(\Odk)$ and $\L(\Ndk)$ directly follows from their unimodularity and the triviality of $\L(\Odk)'\cap\L(\Ndk)$, which gives a new independent proof that $\Ndk$ and $\Odk$ are not of type $\I$, two results obtained in \cite{CapraceLeBoudecMatteBon2022Piecewise} for $\Ndk$ and in \cite{Arimoto2022Type} for $\Odk$.

\Cref{thm:Neretin} is a consequence of the following general criterion of factoriality. Let $H$ be a closed subgroup of $G$ and $K$ be a compact open subgroup of $G$. The normalizer of $K$ in $H$, denoted $\nmlz{H}{K}\doteq\{h\in H\mid hKh^{-1}=K\}$, acts by conjugation on the space of left cosets $G/K$. For $g\in G$, let $\ccc{H}{gK}$ denote the cardinal of the orbit of $gK$ under this action.

\begin{mainthm}[\Cref{the:irreducible inclusion for groups with etoile_H}]\label{mainthm:irreducible inclusion for groups with etoile_H}
    Let $G$ be a totally disconnected locally compact group endowed with a left Haar measure $\mu_G$ and let $\Delta_G:G\to \rr^*_+$ be the modular function of $G$. If $H$ is a closed subgroup of $G$ with the following properties
    \begin{equation*}\tag{$\hspace*{1mm}\star_{H}$} %\label{eq:etoileH}\tag{$\hspace*{1mm}\star_{H}$}
        \left\{
        \begin{array}{l}
        \text{1. }H \leq \ker \Delta_G \\
          \text{2. }\text{There exists a basis at the identity of }G\text{ in compact open subgroups } (K_n)_{n \in \mathbb{N}}\\ 
          \text{ such that for every nontrivial } g \in G,\ \limsup_{n \in \mathbb{N}} \ccc{H}{gK_n} \mu_G(K_n)^2 = \infty,
        \end{array}
        \right.
      \end{equation*}
    then
    \[ \L(H)'\cap \L(G)=\cc.\]
    Thus, both $G$ and $H$ have a factorial left regular representation.
\end{mainthm}

When $G$ is a discrete group and $H=G$, property $\etoile{H}$ is equivalent to the group having infinite conjugacy classes. Apart from Neretin groups, this theorem applies to natural generalizations of the Neretin groups such as \emph{almost automorphism groups associated with regular branch groups }(see \cite[Section 7]{LeBoudec2017CompactPresentability}), \emph{coloured Neretin groups} (see \cite{Lederle2019ColouredNeretin}), as well as to any closed subgroup of $\Ndk$ containing $\Odk$. Along the way, our methods yield several byproducts: we show how they can be applied to prove factoriality for certain $\HNN$ extensions as in \cite{Raum2021Erratum}, and we obtain factoriality results for Hecke von Neumann algebras associated with amalgamated free products, which apply in particular to some Burger–Mozes groups. 

The irreducible inclusion established in \Cref{mainthm:irreducible inclusion for groups with etoile_H} enables us to determine the type of the von Neumann factors under consideration. Most significantly, property $\etoile{H}$ can often be used as a non-type $\I$ criterion, for both $G$ and $H$ (a more precise description of the type of $\L(G)$ is given in \Cref{cor:type of L(G)}.)

\begin{maincor}[\Cref{cor:non typeI factors}]\label{maincor:cor:non typeI factors}
    Let $G$ be a locally compact group. If $G$ has property $\etoile{H}$ with respect to a closed proper subgroup $H$, then both $\L(G)$ and $\L(H)$ are non-type $\I$ factors.
\end{maincor}

In the second part of the article, we leverage property $\etoile{H}$ to obtain factoriality criterion for \emph{crossed products}. Given a continuous action $\alpha:G\acts M$ of a locally compact group on a von Neumann algebra, one can form the associated crossed product $M\rtimes_{\alpha}G$ which contains both $M$ and $\L(G)$ as subalgebras. When $G$ is a discrete group and $M$ is a factor, factoriality of $\L(G)$ implies factoriality of  $M\rtimes_{\alpha}G$, regardless of the action $\alpha$. However, this implication fails for non-discrete groups (see \cite[Proposition B]{Vaes2025factorialitytwistedlocallycompact}). 

\begin{mainthm}[\Cref{the:factoriality for crossed products}]\label{mainthm:factoriality for crossed products}
    Let $G$ be a locally compact group with property \eqref{eq:etoileH} relatively to a closed subgroup $H$ and $\alpha:G\acts (M,\psi)$ be a continuous action
    of $G$ on a von Neumann algebra $M$ endowed with a faithful normal state $\psi$. Let $G_{\psi}=\{g\in G\mid \psi\circ \alpha_g=\psi\}$ be the subgroup of $\psi$-preserving elements. If $G_{\psi}$ is open and $H\leq G_{\psi}$, then
    \[\L(H)'\cap (M\rtimes_{\alpha} G)=M^H,\]
    where $M^H$ is the algebra of $H$-invariant elements in $M$.
    In particular, if $H\acts \Z(M)$ is ergodic, then $M\rtimes_{\alpha}G$ is a factor. 
\end{mainthm}
As an application, we obtain this von Neumann algebraic counterpart of \cite[Theorem B]{Arimoto2025Simplicity}:
 \begin{applnn}[\Cref{Prop:factoriality of the action of Neretin on its boundary}]
    Let $\Ndk$ be a Neretin group, $\nu$ be the visual probability measure on $\bTdk$ and $\alpha:\Ndk\acts \Linf(\bTdk,\nu)$ be the action induced by the non-singular action of $\Ndk$ on $(\bTdk,\nu)$. Then $\Linf(\bTdk,\nu) \rtimes_{\alpha} \Ndk$ is a type $\III_{\frac{1}{d}}$ factor.
\end{applnn}

\textbf{Organization of the paper.} 
The paper is divided into two main parts. 
In \Cref{section:Factoriality Criterion for Group von Neumann Algebras}, we establish all the results concerning group von Neumann algebras. After recalling the relevant notions in \Cref{subsection:group Preliminaries}, we investigate in \Cref{subsection:group Fourier coeffs} the basic properties of \emph{Fourier coefficients} on totally disconnected group von Neumann algebras. Using these Fourier coefficients, we then establish \Cref{mainthm:irreducible inclusion for groups with etoile_H} together with its corollaries in \Cref{subsection: group factoriality}, and apply them to various group families in \Cref{subsection: Hecke algebra of AFP}, \Cref{subsection: Factoriality of HNN extensions} and to the Neretin groups in \Cref{subsection:Neretin}. We emphasize that the proofs of our main results ---\Cref{mainthm:irreducible inclusion for groups with etoile_H}, \Cref{maincor:cor:non typeI factors} and \Cref{mainthm:Neretin}--- are elementary, and require no prior familiarity with von Neumann algebras.

\Cref{section:Factoriality Criterion for Crossed Products} adapts the methods of the first part to the more general setting of crossed product von Neumann algebras. While the core ideas are similar, the proofs are slightly more involved and written in a more operator-algebraic style. 

\tableofcontents

\section{Factoriality Criterion for Group von Neumann Algebras}\label{section:Factoriality Criterion for Group von Neumann Algebras}
    \subsection{Preliminaries}\label{subsection:group Preliminaries}
    \paragraph{Operator topologies.}
    Let $\H$ be a Hilbert space with scalar product denoted by $\ps{\cdot}{\cdot}$ and norm denoted by $\norm{\cdot}$. Let $\B(\H)$ be the algebra of bounded operators on $\H$ and $\U(\H)$ be its subgroup of unitary operators. For any $x\in \B(\H)$, let $\norm{x}_{\infty}$ denote its operator norm. Given vectors $\xi, \eta$ in $\H$, consider the following seminorms on $\B(\H)$:
    \begin{align*}
        \sigma_\xi\;:\; &\B(\H)\to \cc\colon \quad x\mapsto \norm{x\xi}\\
        \omega_{\xi,\eta}\;:\; &\B(\H)\to \cc\colon \quad x\mapsto \abs{\ps{x\xi}{\eta}}.
    \end{align*}
    Then the \emph{strong operator topology} (abbreviated as $\sot$) on $\B(\H)$ is the topology induced by the seminorms $\{\sigma_{\xi}\mid \xi \in \H\}$, and the \emph{weak operator topology} (abbreviated as $\wot$) on $\B(\H)$ is the topology induced by the seminorms $\{\omega_{\xi,\eta}\mid \xi,\eta \in \H\}$. 
    \paragraph{Von Neumann algebras.}
    Let $A \subset \B(\H)$ be a non-empty set. Its \emph{commutant} is the $*$-subalgebra 
    \( A'\doteq \{x\in\B(\H)\mid \forall a\in A, xa=ax\}\) of $\B(\H)$, 
    and  $A''\doteq (A')'$ is the \emph{bicommutant} of $A$. A \emph{von Neumann algebra} $M$ is a unital $*$-subalgebra of some $\B(\H)$ such that $M''=M$. By the von Neumann bicommutant theorem, a unital $*$-subalgebra $M$ of $\B(\H)$ is a von Neumann algebra if and only if it is closed in strong operator topology (equivalently, in weak operator topology). Two von Neumann algebras $M\subset \B(\H)$ and $N\subset \B(\K)$ are isomorphic if there exists an algebraic $*$-isomorphism between the $*$-algebras $M$ and $N$. 
    
    A von Neumann algebra $M$ is a \emph{factor} if its center $M'\cap M$ is trivial, \ie reduced to the scalar multiple of the identity. Factors are classified into types $\I_n$, $n\in \nn^*$, $\I_{\infty}$, $\II_1$, $\II_{\infty}$ and $\III_{\lambda}$, $\lambda\in [0,1]$. A factor is of type $\I_n$ if it is finite dimensional, of type $\I_{\infty}$ if it is infinite dimensional and admits a minimal projection. Any type $\I_n$ factor, with $n\in \nn\cup \{\infty\}$, is of the form $\B(\H_n)\otimes 1_{\K}\subset \B(\H_n\otimes \K)$, where $\H_n$ is a Hilbert space of dimension $n$. It is therefore isomorphic to the factor $\B(\H_n)$. A factor is of type $\II_1$ if it is not of type $\I$ and admits a \emph{trace} (that is a faithful normal tracial \emph{state}), and of type $\II_{\infty}$ if it is not of the previous types, but admits a faithful normal semifinite tracial \emph{weight}. Other factors are of type $\III$, and their further classification into subtypes $\III_{\lambda}$, $\lambda\in [0,1]$, was obtained in \cite{Connes1973Classification} via spectral properties of weights on the factor. For more precise and detailed definitions of these objects, see \cite{GuintoLorentzNelson2025MurrayvonNeumann}.
    
    Let $M\subset\B(\H)$ be a von Neumann algebra, and $p\in M$ be a projection. Then $pMp\subset \B(p\H)$ is again a von Neumann algebra, called \emph{corner} of the von Neumann algebra $M$. 
    
    Let $N\subset M$ be an inclusion of von Neumann algebras. A \emph{conditional expectation} $E:M\to N$ is a projection of norm one between the Banach space $M$ and $N$. A von Neumann algebra $M\subset\B(\H)$ is \emph{amenable} if there exists a conditional expectation $E:\B(\H)\to M$. Amenable von Neumann algebras are very well understood compared to the non-amenable ones since Connes and Haagerup have completely classified the amenable factors. In particular, there is only one isomorphism class of amenable factors of types $\II_1$, $\II_{\infty}$, and $\III_{\lambda}$, $\lambda>0$, such factors being called \emph{hyperfinite}. If $G$ is amenable as a group, then $\L(G)$ is amenable. The converse holds for discrete groups, but not in general.

    \paragraph{Locally compact groups.}
    Let $G$ be a locally compact group. We denote by $\mu_G$ a left Haar measure on $G$, and by $\Delta_G$ its modular function defined by the relation $\mu_G(Ag)=\Delta_G(g)\mu_G(A)$ for every Borel subset $A\subset G$. We denote by $\Ldeux(G,\mu_G)$ the Hilbert space of square-integrable functions on $G$ with respect to $\mu_G$. Subgroups $H \leq G $ will always be assumed to be closed. We write $K\co G$ to indicate that $K$ is a \emph{compact open} subgroup of $G$. 
    
    If $G$ is \emph{totally disconnected}, we denote by $\cosub{G}$ the directed set of compact open subgroups of $G$, and by $\lcocosets{G}$ the directed set of cosets of compact open subgroups: $\lcocosets{G}=\{gK\mid K\co G, g\in G\}=\{Kg\mid K\co G, g\in G\}$. By van Dantzig theorem, $\cosub{G}$ is a basis of neighborhood at the identity of $G$.

    \paragraph{Unitary representations.}
    A unitary representation is a pair $(\pi,\H)$ consisting of a Hilbert space $\H$ and a continuous homomorphism $\pi:G\to (\U(\H),\sot)$. We will mostly consider the left and right regular representations defined for $\xi\in \Ldeux(G,\mu_G)$, $g,x\in G$ by
    \begin{align*}
        &\lambda_G:G\to\U(\Ldeux(G,\mu_G))\colon\quad \lambda_G(g)\xi(x)=\xi(g^{-1}x)\\
        &\rho_G:G\to\U(\Ldeux(G,\mu_G))\colon\quad \rho_G(g)\xi(x)=\Delta_G(g)^{1/2}\xi(xg).
    \end{align*}
    When the group is clear from the context we write $\lambda_g$ and $\rho_g$ for $\lambda_G(g)$ and $\rho_G(g)$ respectively.

    \paragraph{Group von Neumann algebras.}

    Let $G$ be a locally compact group. The \emph{group von Neumann algebra} of $G$, denoted by $\L(G)$, is the von Neumann algebra $\lambda_G(G)''\subset \B(\Ldeux(G,\mu_G))$. By the von Neumann bicommutant theorem, 
    \[ \L(G)=\overline{\vect\{\lambda_g\mid g\in G\}}^{\sot}=\overline{\vect\{\lambda_g\mid g\in G\}}^{\wot}.\]
    Since $\lambda_G$ and $\rho_G$ are unitarily equivalent (\cite[Proposition A.4.1]{BekkadelaHarpeValette2008KazhdansPropertyT}), $\L(G)\simeq \rho_G(G)''$. We even have $\L(G)=\rho_G(G)'$ (\cite[Proposition VII.3.1]{Takesaki2003TheoryOperator}): while the inclusion $\L(G)\subset \rho_G(G)'$ is obvious, the reverse inclusion is a non-trivial result of modular theory. 

    Every group von Neumann algebra $\L(G)$ is endowed with endowed its \emph{Plancherel weight} $\phi_G$ (\cite[VII.3]{Takesaki2003TheoryOperator}). This is a normal semifinite faithful weight, characterized by $\phi_G(\lambda_G(f))=f(e)$ for every continuous function $f$ with compact support on $G$. When $G$ is unimodular, this weight is moreover tracial, and it is finite if and only if $G$ is discrete.

    Given a closed subgroup $H\cl G$, the group von Neumann algebra $\L(H)$ is isomorphic to the subalgebra $\lambda_G(H)''\subset \L(G)$(see \cite[Proposition 2.1]{Miyamoto2025PlancherelWeight}), and therefore we identify them. For example, the expression $\L(H)'\cap \L(G)$ denotes the algebra of elements of $\L(G)$ that commutes with $\{\lambda_G(h)\mid h\in H\}$. 
    
    Given a compact open subgroup $K$ of $G$, we identify $\ldeux(G/K)$ with the Hilbert subspace $\Ldeux(G/K)$ of $\Ldeux(G,\mu_G)$ consisting of right $K$-invariant vectors, and similarly for $\ldeux(K\backslash G)$. We denote by $p_K$ the \emph{left averaging projection} associated to $K$: it is the orthogonal projection onto $\Ldeux(K\setminus G)$, and it belongs to $\L(G)$. Given a pair $(G,K)$ where $K$ is a compact open subgroup of a locally compact group $G$, the \emph{Hecke von Neumann algebra} of the pair is the von Neumann algebra $p_K\L(G)p_K$. By  \cite[Theorem 4.2]{Tzanev2003Hecke}, this definition coincides with the usual one, obtained as the appropriate closure of the algebra of compactly supported continuous $K$-biinvariant functions on $G$. By \cite[Theorem 4.5]{LandstadLarsen2009Generalized}, the Hecke von Neumann algebra $p_K\L(G)p_K$ is the von Neumann algebra associated to the \emph{quasi-regular representation} $\lambda_{G/K}$, that is the von Neumann algebra $\lambda_{G/K}(G)'\subset \B(\ldeux(G/K))$. Factoriality result for Hecke von Neumann algebras can therefore be interpreted as factoriality results for quasi-regular representations.

    Finally, we recall the following classical fact about averaging projections in totally disconnected group von Neumann algebra.  For any $A\subset \lcocosets{G}$, $\xi_A\doteq \1_A/{\mu_G(A)}^{1/2}$ will denote the normalized indicator function of $A$ in $\Ldeux(G,\mu_G)$. 

    \begin{Prop}\label{lem:sot.conv.of p_K}
        Let $G$ be a totally disconnected locally compact group. Then $(p_K)_{K\in \cosub{G}}$ $\sot$-converges to the identity.
    \end{Prop}
    \begin{proof}
    Consider \[ \E\doteq \vect\{\xi_{gK}\mid gK\in \lcocosets{G}\}\subset \Ldeux(G,\mu_G).\]
    Note that $\E$ is dense in $\Ldeux(G,\mu_G)$ because $G$ is totally disconnected. 
    Then, given any $\eta \in \E$, there exists $K_0\in \cosub{G}$ such that for every $K$ contained in $K_0$ we have $p_K \eta=\eta$.  Therefore, $p_K$ strongly converges to the identity as the group $K$ shrinks to $\{e\}$.
    \end{proof}

    \subsection{Fourier Coefficients on Group von Neumann Algebras}\label{subsection:group Fourier coeffs}
    
    In this section, we introduce an adapted notion of \emph{Fourier coefficients} for locally compact groups von Neumann algebras. Let $G$ be a locally compact group. 
    \begin{Def} 
        Given $x\in \L(G)$, $g\in G$ and $K$ a compact open subgroup of  $G$, the \emph{Fourier coefficient of $x$ at the coset $gK$} is the complex number
            \[\fc{x}{gK}\doteq\ps{x\xi_{K}}{\xi_{gK}}.\]
        We therefore consider $\fc{x}{}$ as a map from $\lcocosets{G}$ to $\cc$.
    \end{Def}
    
    Let us establish some basic properties of these Fourier coefficients. \Cref{lem:continuity properties of Fourier coefficients for crossed products} states their basic continuity properties, \Cref{lem:monotonicity of phi/mu} describes their behavior under the replacement of $K$ by a compact open subgroup $K'\co K$, while \Cref{lem:caracterization of scalar elements of p_KL(G)p_K by Fourier coefficients} and \Cref{lem:caracterization of scalar elements by Fourier coefficients} establish their definiteness properties when $G$ is totally disconnected.
    \begin{Prop}\label{lem:continuity properties of Fourier coefficients for crossed products} Let $x\in \L(G)$, $g\in G$ and $K$ be a compact open subgroup of $G$.
        \begin{enumerate}
            \item $\fc{x}{}:\lcocosets{G}\ni gK\mapsto \fc{x}{gK}$ is bounded by $\norm{x}_{\infty}$.
            \item $\L(G)\ni x\mapsto \fc{x}{gK} \in \cc$ is linear and weakly continuous.
        \end{enumerate}
    \end{Prop}
    \begin{proof}
        The first item is a direct consequence of the Cauchy-Schwarz inequality. The second one is a consequence of the fact that the absolute value of the map $x\mapsto \fc{x}{gK}$ coincides with $\omega_{\xi_{K},\xi_{gK}}$ which is one of the seminorms defining the weak operator topology.
    \end{proof}
    \begin{Prop}\label{lem:monotonicity of phi/mu}
        Let $x\in \L(G)$, $g\in G$ and $K, K'$ be compact open subgroups of $G$ such that $K'\leq K$. Pick $g_1,\dots, g_n$ in $gK$ such that $gK=\bigsqcup_{i=1}^ng_iK'$. Then
            \[\fc{x}{gK}=\sum_{i=1}^n\fc{x}{g_iK'}.\]
    \end{Prop}
    \begin{proof}
      Let $\gamma\in G$: we first prove the result for $x=\lambda_{\gamma}$.
            \begin{align*}
            \sum_{i=1}^n\fc{\lambda_{\gamma}}{g_iK'}
            &=\sum_{i=1}^n\int_G\xi_{\gamma K'}(h)\xi_{g_iK'}(h)d\mu_G(h)
            =\frac{1}{\mu_G(K')}\sum_{i=1}^n\mu_G(\gamma K'\cap g_iK')\\
            &=\frac{1}{\mu_G(K')}\mu_G(\gamma K'\cap gK)
            =\frac{1}{\mu_G(K)}\mu_G(\gamma K\cap gK)
            =\fc{\lambda_{\gamma}}{gK}.
            \end{align*} 
        By linearity, the equality holds for any $x\in \vect\{\lambda_g\mid g\in G\}$. This space is $\wot$-dense in $\L(G)$; therefore, the $\wot$-continuity of Fourier coefficients conclude.
    \end{proof}

    \begin{Prop}\label{lem:caracterization of scalar elements of p_KL(G)p_K by Fourier coefficients}
        Let $x\in \L(G)$ and $K$ be a compact open subgroup of $G$. Then
            \[\fc{x}{}|_{G/K}=0\text{ if and only if }xp_K=0.\]
        As a consequence, $x$ is a scalar multiple of $p_K$ if and only if for all $g\in G\setminus K$, $\fc{x}{gK}=0$.
    \end{Prop}

    \begin{proof}
        Let $x\in \L(G)$ such that $xp_K=0$. Then   
        \[\fc{x}{gK}=\ps{x\xi_K}{\xi_{gK}}=\ps{xp_K\xi_K}{\xi_{gK}}=0.\]
        Conversely, assume that $x$ is such that for every $g\in G$, $\fc{x}{gK}=0$. By sesquilinearity of the scalar product and density of $\vect\{\xi_A\mid A\in G/K\}$ in $\Ldeux(G/K)$, $\ps{x\xi_K}{\eta}=0$ for any $\eta\in \Ldeux(G/K)$. Therefore, $x\xi_K=0$, because $\L(G)$ stabilizes the subspace $\Ldeux(G/K)\ni \xi_K$. By applying $\rho_g$ for $g\in G$ on the last equality we get that for all $\eta\in \Ldeux(K\setminus G)$, $x\eta=0$. Therefore, because $\Ldeux(K\setminus G)$ is precisely the range of $p_K$, we get that $xp_K=0$.

        The second part of the proposition follows: it is clear that the Fourier coefficients of scalar multiples of $p_K$ are zero except maybe in cosets containing the identity. Conversely, if for every $g\notin K$, $\fc{x}{gK}=0$, then the Fourier coefficients of $x-\fc{x}{K}\id$ are zero on $G/K$, and by the previous statement, $xp_K=\fc{x}{K}p_K$ is a scalar multiple of $p_K$.
    \end{proof}
            
    \begin{Prop}\label{lem:caracterization of scalar elements by Fourier coefficients}
        Assume that $G$ is a totally disconnected locally compact group. Then, for any $x\in \L(G)$,
        \[\fc{x}{}=0\text{ if and only if }x=0.\]
        As a consequence, $x$ is scalar if and only if for every compact open subgroup $K$ of $G$ and for every $g\notin K$, $\fc{x}{gK}=0$.
    \end{Prop}

    \begin{proof}
        Assume that every Fourier coefficient of $x$ is zero. It follows from \Cref{lem:caracterization of scalar elements of p_KL(G)p_K by Fourier coefficients} that for any $K\co G$, $xp_K$ is zero.
        On the other hand, by \Cref{lem:sot.conv.of p_K}, $xp_K$ $\sot$-converges to $x$ as $K$ shrinks to the identity. Therefore, $x$ is zero. The other direction is trivial.

        For the second part, assume that $x$ is such that for all $K\co G$, $\fc{x}{gK}=0$ when $g\notin K$. The previous lemma shows that for all $K\co G$, $xp_K=\fc{x}{K}p_K$. The net $(xp_K)_{K\in \cosub{G}}$ $\sot$-converges to $x$, thus its norm is bounded. Therefore, $(\fc{x}{K})_{K\in \cosub{G}}$ itself is bounded and admits an accumulation point $\alpha\in \cc$. Then $\alpha\cdot\id$ is an $\sot$-accumulation point of $(\fc{x}{K}p_K)_{K\in \cosub{G}}$. As this net $\sot$-converges to $x$, we get that $x$ is a scalar multiple of the identity.  
    \end{proof}

    \subsection{Factoriality Criterion}\label{subsection: group factoriality}
    We can now establish a locally compact adapted version of the inequality used in the discrete case to show that infinite conjugacy class groups have factorial group von Neumann algebras. Given $G$ a locally compact group, $K$ a compact open subgroup of $G$ and $H$ a closed subgroup of $G$, $\nmlz{G}{K}\doteq \{g\in G\mid gKg^{-1}=K\}$ is the normalizer of $K$ in $G$ and we set $\nmlz{H}{K}\doteq H\cap \nmlz{G}{K}$. The group $\nmlz{H}{K}$ acts by conjugation on $G/K$.  

    \begin{Lem}\label{lem:hilbert inequality}
        Let $H$ be a closed subgroup of a locally compact group $G$, and $z\in \L(H)'\cap\L(G)$. If $H\leq\ker \Delta_G$, then for every  $g\in G$ and for every compact open subgroup $K$ of $G$,
            \[\norm{z}_{\infty}^2\geq \ccc{H}{gK}\abs{\fc{z}{gK}}^2, \]
        where $\ccc{H}{gK}\in \nn^*\cup \{\infty\}$ is the cardinal of the orbit of $gK$ for the conjugation action of $\nmlz{H}{K}$ on $G/K$.
    \end{Lem}

    \begin{proof}
        Given $gK\in G/K$, denote by $\Ccc{H}{gK}$ the orbit of $gK$ under the conjugation action of $\nmlz{H}{K}$ on $G/K$.

         The family $(\xi_{B})_{B\in \Ccc{H}{gK}}\in \Ldeux(G,\mu_G)$ is orthonormal. Given any $B\in\Ccc{H}{gK}$, let $h_B\in \nmlz{H}{K}$ such that $B=h_B(gK)h_B^{-1}$. Then
        \begin{align*}
            \norm{z}_{\infty}^2&
            \geq \norm{z\xi_{K}}_2^2
            \geq \sum_{B\in \Ccc{H}{gK}}\abs{\ps{z\xi_{K}}{\xi_{B}}}^2
            = \sum_{B\in \Ccc{H}{gK}}\abs{\ps{z\xi_{K}}{\xi_{h_B(gK)h_B^{-1}}}}^2\\
            &= \sum_{B\in \Ccc{H}{gK}}\abs{\ps{z\xi_{K}}{\lambda_{h_B}\rho_{h_B}\xi_{gK}}}^2
            = \sum_{B\in \Ccc{H}{gK}}\abs{\ps{z\rho_{h_{B}^{-1}}\lambda_{h_{B}^{-1}}\xi_{K}}{\xi_{gK}}}^2\tag{because $z\in \L(H)'$}\\
            &= \sum_{B\in \Ccc{H}{gK}}\abs{\ps{z\xi_{h_{B}^{-1}Kh_{B}}}{\xi_{gK}}}^2
            = \sum_{B\in \Ccc{H}{gK}}\abs{\ps{z\xi_{K}}{\xi_{gK}}}^2\\
            &= \ccc{H}{gK}\abs{\fc{z}{gK}}^2.
        \end{align*}
   \end{proof}

   \Cref{lem:hilbert inequality} implies the following factoriality criterion for Hecke von Neumann algebras.

\begin{Prop}\label{prop:factoriality for corners}
    Let $K$ be a compact open subgroup of a locally compact group $G$ and $H$ be a closed subgroup of $\nmlz{G}{K}\cap \ker \Delta_G$. If for any $g\notin K$, $\ccc{H}{gK}=\infty$, then 
    \[(p_K\L(H))'\cap p_K\L(G)p_K=\cc p_K.\]
    In particular, $p_K\L(G)p_K$ is a factor. 
\end{Prop}

\begin{proof}
 Let $z\in (p_K\L(H))'\cap p_K\L(G)p_K$. As $H\leq \nmlz{G}{K}$, $p_K\in \L(H)'\cap \L(G)$. Therefore, $z\in \L(H)'\cap \L(G)$. By \Cref{lem:hilbert inequality}, $\ccc{H}{gK}$ being infinite for all $g\notin K$, it follows that $\fc{z}{gK}=0$ for every $g\notin K$. By \Cref{lem:caracterization of scalar elements of p_KL(G)p_K by Fourier coefficients}, $z\in \cc p_K$.
\end{proof}
We can now prove our main factoriality criterion. 

\begin{The}[\Cref{mainthm:irreducible inclusion for groups with etoile_H}]\label{the:irreducible inclusion for groups with etoile_H}
      Let $G$ be a totally disconnected locally compact group endowed with a left Haar measure $\mu_G$ and let $\Delta_G:G\to \rr^*_+$ be the modular function of $G$. If $H$ is a closed subgroup of $G$ with the following properties
    \begin{equation*} \label{eq:etoileH}\tag{$\hspace*{1mm}\star_{H}$}
        \left\{
        \begin{array}{l}
        \text{1. }H \leq \ker \Delta_G \\
          \text{2. }\text{There exists a basis at the identity of }G\text{ in compact open subgroups } (K_n)_{n \in \mathbb{N}} \\
          \text{ such that for every nontrivial } g \in G,\ \limsup_{n \in \mathbb{N}} \ccc{H}{gK_n} \mu_G(K_n)^2 = \infty,
        \end{array}
        \right.
      \end{equation*}
    then
    \[ \L(H)'\cap \L(G)=\cc.\]
\end{The}
\begin{proof}
    We first establish the following consequence of \Cref{lem:monotonicity of phi/mu}: given $K'\co K\co G$, $x\in \L(G)$ and $g\in G$, there exists $g'\in gK$ such that 
        \(\frac{\abs{\fc{x}{g'K'}}}{\mu_G(K')}\geq \frac{\abs{\fc{x}{gK}}}{\mu_G(K)}.\)
    Indeed, if $gK=\bigsqcup_{i=1^n}g_iK'$:
        \[\frac{\abs{\phi^x(gK)}}{\mu_G(K)}
        =\abs{ \sum_{i=1}^n \frac{\phi^x(g_iK')}{\mu_G(K)}}
        =\frac{1}{n}\abs{ \sum_{i=1}^n \frac{\phi^x(g_iK')}{\mu_G(K')}}
        \leq \frac{1}{n}\sum_{i=1}^n \frac{\abs{\phi^x(g_iK')}}{\mu_G(K')}
        \leq \max \frac{\abs{\phi^x(g_iK')}}{\mu_G(K')}.\]

    We now proceed to the proof. Let $z\in \L(G)\cap \L(H)'$. Assume by contradiction that there exists $K\co G$ and $g\notin K$ such that $\fc{z}{gK}\neq 0$. Consider $(K_n)_{n\in \nn}$ a basis at the identity in compact open subgroups with respect to which $G$ has \eqref{eq:etoileH}. Using repeatedly the inequality established above, we obtain $n_0\in \nn$ and a sequence $(g_n)_{n\geq n_0}$ such that $g_{n_0}K_{n_0}\subset gK$ and for every $n\geq n_0$, we have the inclusion
    $g_{n+1}K_{n+1}\subset g_{n}K_n$ and the inequality $\frac{\abs{\fc{z}{g_{n+1}K_{n+1}}}}{\mu_G(K_{n+1})}\geq \frac{\abs{\fc{z}{gK}}}{\mu_G(K)}$. The sequence $(g_n)_{n\geq n_0}$ converges. Let $h$ be its limit: it belongs to $gK$ and is therefore in $G\setminus\{e\}$.  \Cref{lem:hilbert inequality} gives 
    \[ \norm{z}_{\infty}^2
    \geq \ccc{H}{hK_n}\abs{\fc{z}{hK_n}}^2
    =\ccc{H}{hK_n}\abs{\fc{z}{g_nK_n}}^2
    \geq \ccc{H}{hK_n}\mu_G(K_n)^2\frac{\abs{\fc{z}{gK}}^2}{\mu_G(K)^2}.
    \]

    Therefore, $\limsup \ccc{H}{hK_n}\mu_G(K_n)^2=\infty$, and thus $\fc{z}{gK}=0$, leading to a contradiction. By \Cref{lem:caracterization of scalar elements by Fourier coefficients}, $z$ is scalar.
\end{proof}
Note that whenever $H\leq H'\leq \ker \Delta_G$, property $\etoile{H}$ implies property $\etoile{H'}$. Furthermore, if $G$ has property $\etoile{H}$ with respect to an \emph{open} subgroup $H$, and if $G'$ is such that $H\leq G'\leq G$, then $G'$ also has property $\etoile{H}$.

As a first corollary, we obtain a slightly generalized version of the von Neumann algebraic version of \cite{Suzuki2022Csimplicity}:
\begin{Cor}\label{cor:generalizedSuzukicriterion}
    Let $G$ be a totally disconnected group. Assume that there exists a basis at the identity in compact open neighborhoods $(K_n)_{n\in \nn}$ such that for all $g\in G\setminus\{e\}$, there is arbitrary large $n\in \nn$ such that $\ccc{G}{gK_n}=\infty$. Then $\L(G)$ is a factor. 
\end{Cor}

We now determine the type of the factors obtained \Cref{the:irreducible inclusion for groups with etoile_H}. 

\begin{Cor}[\Cref{maincor:cor:non typeI factors}]\label{cor:non typeI factors}
    Let $G$ be a locally compact group. If $G$ has property $\etoile{H}$ with respect to a closed proper subgroup $H$, then both $\L(G)$ and $\L(H)$ are non-type $\I$ factors.
\end{Cor}
\begin{proof}
    Let $A\subset B$ be an inclusion of von Neumann algebra such that $A'\cap B=\cc$ (we say that the inclusion is \emph{irreducible}). Assume that $B$ is a factor of type $\I$. Then $A=(A'\cap B)'\cap B=B$. Therefore, if $A\neq B$, then $B$ cannot be of type $\I$. Moreover, $B'\subset A'$ is also an irreducible inclusion of von Neumann algebras because $(B')'\cap A'=A'\cap B=\cc$. Therefore, the same argument implies that if $A'$ is a type $\I$ factor, then $B'=A'$ and thus $A=B$. Therefore, if $A\neq B$ is a proper von Neumann subalgebra irreducibly contained in $B$, then $A'$ cannot be of type $\I$, and $A$ neither as taking the commutant preserves type of factors. By \Cref{the:irreducible inclusion for groups with etoile_H}, as soon as $H\neq G$ and $G$ has $\etoile{H}$, we are in this situation and therefore both $\L(H)$ and $\L(G)$ are non-type $\I$ factors.
\end{proof}

Note that property $\etoile{H}$ for any closed subgroup $H$ of $G$ implies that $\L(\ker \Delta_G)$ is a factor. This, in turn, allows for an explicit computation of the $\Sinv$-invariant $\Sinv(\L(G))$ of the factor $\L(G)$, an invariant that was introduced by Connes in \cite{Connes1973Classification} to classify type $\III$ factors (see also \cite{GuintoLorentzNelson2025MurrayvonNeumann} for a concise introduction to this subject). 

\begin{Cor}\label{cor:type of L(G)}
Let $G$ be a locally compact group. If $G$ has property $\etoile{H}$ with respect to a closed subgroup $H$, then $\Sinv(\L(G))=\overline{\Delta_G(G)}$. Therefore,
\begin{itemize}
    \item $\L(G)$ is a type $\II_1$ factor if $G$ is discrete,
    \item $\L(G)$ is a type $\I_{\infty}$ or $\II_{\infty}$ factor if $G$ is unimodular and non-discrete,
    \item $\L(G)$ is a type $\III_{\lambda}$ factor for $\lambda\in ]0,1[$ if $\Delta_G(G)=\lambda^{\zz}$,
    \item $\L(G)$ is a type $\III_{1}$ factor if $\Delta_G(G)$ is dense in $\rr^{*}_+$.
\end{itemize}
\end{Cor}
\begin{proof}
    Let $\phi_G$ denote the Plancherel weight on $\L(G)$. The centralizer $\L(G)_{\phi_G}$ of $\phi_G$ is precisely the von Neumann algebra $\L(\ker \Delta_G)$, and property $\etoile{H}$ implies that $\L(\ker \Delta_G)$ is a factor. Therefore, by \cite[Corollaire 3.2.7]{Connes1973Classification}, $\Sinv(\L(G))=\Sp(\Delta_{\phi_G})$,  where $\Delta_{\phi_G}$ is the \emph{modular operator} associated to the weight $\phi_G$, and $\Sp(\Delta_{\phi_G})$ denotes its spectrum. But it turns out that, given any locally compact group $G$, $\Sp(\Delta_{\phi_G})=\overline{\Delta_G(G)}$ (see \cite[Proposition 2.26]{Raum2019Csimplicity}).  
    Thus, $\Sinv(\L(G))=\overline{\Delta_G(G)}$. If $G$ is discrete, $\L(G)$ is a type $\II_1$ as $\phi_G$ is finite and tracial. If $G$ is unimodular and non-discrete, its Plancherel weight is tracial and infinite, and thus it is of type $\I_{\infty}$ or $\II_{\infty}$. The remaining cases follow from the definition of types $\III_{\lambda}$ in terms of the $\Sinv$-invariant.
\end{proof}

\subsection{Application 1: Hecke von Neumann Algebras of Amalgamated Free Products}\label{subsection: Hecke algebra of AFP}

    In this section, we apply \Cref{prop:factoriality for corners} to amalgamated free products of locally compact groups (see \cite[Chapter 8.B]{CornulierHarpe2016MetricGeometry}). This yields the factoriality of the Hecke von Neumann algebra naturally associated with certain Burger-Mozes groups, defined in \cite{BuergerMozes2000GroupsActing}. Before that, we set up some notations for group actions. Given $G\acts X$ an action of a group on a given set and $A$ a subset of $X$, we use the following notations for the fixator, stabilizer, and set of fixed points of the action:
        \begin{align*}
            \Fix_{G}(A)&\doteq \{ g\in G \mid \forall x\in A, gx=x\}\\
            \Stab_{G}(A)&\doteq \{ g\in G \mid \forall x\in A, gx\in A\}\\
            \fp{A}{G}&\doteq \{x\in A\mid \forall g\in G, gx=x\}.
        \end{align*}

    \begin{Prop}\label{prop:factoriality of corner of AFP}
        Let $A,B$ be two locally compact groups with a common compact open subgroup $K$, and let $G=A\ast_{K}B$. If $\abs{\nmlz{A}{K}/K}\geq 3$ and $\abs{\nmlz{B}{K}/K}\geq 2$, then $p_KL(G)p_K$ is a non-amenable factor.
    \end{Prop}
    \begin{proof}
        We first explain how to drop the unimodularity assumption that appears in \Cref{prop:factoriality for corners}. By construction, $A$ and $B$ are open subgroups of $G$, so the restrictions of $\Delta_G$ give their modular functions. The same holds for $\nmlz{A}{K}$ and $\nmlz{B}{K}$ because $K$ is open. Let $\nmlz{A}{K}^{0}=\nmlz{A\cap \ker \Delta_G}{K}$ denote the unimodular part of $\nmlz{A}{K}$. If $\abs{\nmlz{A}{K}^{0}/K}\leq 2$, then $\nmlz{A}{K}^{0}$ is a compact open normal subgroup of $\nmlz{A}{K}$. Therefore, $\nmlz{A}{K}=\nmlz{A}{K}^0$ and thus $\abs{\nmlz{A}{K}/K}\leq 2$, hence a contradiction. Therefore, we may assume that $\abs{\nmlz{A}{K}^0/K}\geq 3$, and the same line of reasoning ensures that we can assume  $\abs{\nmlz{B}{K}^0/K}\geq 2$. 
        
        Let $H$ be the open subgroup $\nmlz{A}{K}^0\ast_{K}\nmlz{B}{K}^0$ of $G$ : $H$ is contained both in $\ker \Delta_G$ and in $\nmlz{G}{K}$. Let $a_1,a_2\in \nmlz{A}{K}^0\setminus K$ such that $a_2^{-1}a_1\notin K$, and $b\in \nmlz{B}{K}^0\setminus K$. Let $g\in G\setminus K$: it can be (uniquely) written as a product 
        \[g=g_1g_2\cdots g_nk\quad \text{ where }n\geq 1, k\in K,\text{ and  the } g_i\text{'s alternatively belong to $A\setminus K$ and $B\setminus K$}.\]
        We claim that $\ccc{H}{gK}=\infty$. This follows from the following case analysis. If
        \begin{description}
            \item[{\boldmath\bfseries$g_1\in A$, $g_n\in A$,}] then $\{(a_1b)^n(gK)(a_1b)^{-n}\}_{n\in \nn}$ is an infinite family of mutually disjoint cosets in the orbit of $gK$ under the conjugation action. 
            \item[{\boldmath\bfseries$g_1\in B$, $g_n\in B$,}] then $\{(ba_1)^n(gK)(ba_1)^{-n}\}_{n\in \nn}$ is an infinite family of mutually disjoint cosets in the orbit of $gK$ under the conjugation action. 
            \item[{\boldmath\bfseries$g_1\in A$, $g_n\in B$,}] then there exists $i\in \{1,2\}$ such that $a_ig_1\notin K$, and then $\{(ba_i)^n(gK)(ba_i)^{-n}\}_{n\in \nn}$ is an infinite family of mutually disjoint cosets in the orbit of $gK$ under the conjugation action. 
            \item[{\boldmath\bfseries$g_1\in B$, $g_n\in A$,}] then there exists $i\in \{1,2\}$ such that $a_ig_n\notin K$, and then $\{(a_ib)^n(gK)(a_ib)^{-n}\}_{n\in \nn}$ is an infinite family of mutually disjoint cosets in the orbit of $gK$ under the conjugation action.
        \end{description}
        \Cref{prop:factoriality for corners} then implies the factoriality of $p_K\L(G)p_K$. The non-amenability of $p_K\L(G)p_K$ is a consequence of \cite[Lemma 3.1]{HoudayerRaum2019LocallyCompact}.
    \end{proof}

    We now apply this criterion to groups acting on trees. Given any graph $\G$, denote by $\V(\G)$ its set of vertices, and by $\E(\G)$ its set of edges. Given any vertex $v\in \V(\T_d)$, denote by $\E_v(\T_d)$ the set of edges adjacent to $v$. Let $\T$ be a tree, and denote by $\Aut(\T)$ its totally disconnected locally compact group of automorphisms.
    By Bass-Serre theory, every closed subgroup of $\Aut(\T)$ acting on $\T$ without inversions, with two orbits of vertices and one orbit of edges, is isomorphic to an amalgamated free product of compact totally disconnected groups (see \cite[Section 2.4]{HoudayerRaum2019LocallyCompact}). We consider the Burger-Mozes group introduced in \cite{BuergerMozes2000GroupsActing} as an application. 
    
    \paragraph{Burger-Mozes groups.}  
    Fix $d\geq 3$ and let $\T_d$ be the $d$-regular tree. A \emph{legal coloring} of $\T_d$ is a map $\operatorname{col}:\E(\T_d)\to \{1,\dots,d\}$ such that for every $v\in \V(\T_d)$, $\operatorname{col}|_{ \E_v(\T_d)}:\E_v(\T_d)\to \{1,\dots,d\}$ is a bijection. Given $g\in\Aut(\T_d)$, the \emph{local action of $g$} at the vertex $v$ is the permutation $\sigma(g,v)\in \Sym_d$ defined by
    \( \sigma(g,v)\doteq \operatorname{col}\circ g \circ (\operatorname{col}|_{\E_v(\T_d)})^{-1}.\)
    Given a subgroup $F$ of $\Sym_d$, the \emph{Burger-Mozes group associated to $F$} is defined as 
    \[ \bm{F}\doteq \{g\in \Aut(\T_d)\mid \sigma(g,v)\in F \text{ for every } v\in \V(\T_d)\}.\]
    These groups do not depend on the choice of the coloring. They always act transitively on $\V(\T_d)$, albeit possibly with inversions. Let $\Aut(\T_d)^+$ denote the index $2$ subgroup of $\Aut(\T_d)$ consisting of \emph{type preserving} automorphisms: these are the automorphisms that preserve the canonical bipartition of the vertices of the tree. Then $\bmp{F}\doteq \bm{F}\cap \Aut(\T_d)^+$ is an index two open subgroups of $\bm{F}$, that acts without inversion on $\T_d$. The groups $\bm{F}$ and $\bmp{F}$ are discrete if and only if $F$ acts freely on $\{1,\dots,d\}$; that is, if $\Fix_F(\{k\})=\{\id\}$ for every $k\in \{1,\dots,d\}$. 

    In the next corollary, we obtain factoriality for some Hecke von Neumann algebras of Burger-Mozes groups. In this situation, we are able to describe this von Neumann algebra quite explicitly, thanks to Dykema's general results on amalgamated free products of multimatrix von Neumann algebra. We therefore introduce the so called \emph{interpolated free group factors}. 

    \paragraph{ Interpolated free group factors} Given $n\geq 2$, let $\ff_n$ denote the free group with $n$ generators. As $\ff_n$ has infinite conjugacy classes, $\L(\ff_n)$ is a factor called a \emph{free group factor}. A continuous interpolation $(\L(\ff_t))_{t>1}$ of this sequence of factors has been independently constructed in \cite{Radulescu1994RandomMatrices} and \cite{Dykema1994Interpolated} under the name of \emph{interpolated free group factors}: $\L(\ff_t)$ is isomorphic to any corner $p\M_k(\cc)\ptvn \L(\ff_n)p$ where $p$ is a projection in $\M_k(\cc)\ptvn \L(\ff_n)$ with $\operatorname{Tr}(p)^2=\frac{n-1}{t-1}$, $\operatorname{Tr}$ being the tensor product of the normalized traces on $\M_k(\cc)$ and $\L(\ff_n)$. These factors are non-amenable.
    
    \begin{Cor}\label{cor:Hecke Burger Mozes}
        Let $F\leq\Sym_d$, and $K=\Fix_{\bmp{F}}(\{e\})\leq \bmp{F}$ for a given $e\in \E(\T_d)$. If $F\acts \{1,\dots,d\}$ is transitive and $\abs{\fp{\{1,\dots,d\}}{\Fix_F(\{1\})}}\geq 3$, then 
        \(p_{K}\L(\bmp{F})p_{K}\) is a factor.
        More precisely, $p_{K}\L(\bmp{F})p_{K} \simeq \L(\ff_t)$ for some $t>0$.  
    \end{Cor}
    \begin{proof}
        Let $v_1,v_2$ be the two endpoints of $e$. Without loss of generality, we may assume that $\operatorname{col}(e)=1$. Set $G=\bmp{F}$, $H=\Fix_G(v_1)\simeq \Fix_G(v_2)$, $K=\Fix_G(e)$ and $F_k=\Fix_F(\{k\})$ for every $k\in \{1,\dots,d\}$. By Bass-Serre theory, $G$ is isomorphic to $H\ast_K H$. 
        Moreover,  
        \(\nmlz{H}{K}=\Stab_H(\fp{\E_{v_0}(\T_d)}{K})),\)
        so that 
        \(\nmlz{H}{K}/K\simeq \Stab_F(\fp{\{1,\dots,d\}}{F_1})/F_1.\)
        By transitivity of $F$, $k\in \fp{\{1,\dots,d\}}{F_1}$ if and only if $F_k=F_1$. Let $k,l\in \fp{\{1,\dots,d\}}{F_1}$, and $\sigma\in F$ such that $\sigma(k)\in\fp{\{1,\dots,d\}}{F_1}$. Then $F_{\sigma(l)}=\sigma F_l \sigma^{-1}=\sigma F_k \sigma^{-1}=F_{\sigma(k)}=F_1$ and thus $\sigma(l)\in \fp{\{1,\dots,d\}}{F_1}$. Therefore, $\sigma\in \Stab_H(\fp{\{1,\dots,d\}}{F_1})$ if and only if $\sigma(1)\in \fp{\{1,\dots,d\}}{F_1}$. By transitivity of $F$, $\abs{\nmlz{H}{K}/K}=\abs{\Stab_H(\fp{\{1,\dots,d\}}{F_1})/F_1}\geq \abs{\fp{\{1,\dots,d\}}{F_1}}\geq 3$, and therefore $p_K\L(\bmp{F})p_K$ is a factor by \Cref{prop:factoriality of corner of AFP}.

        This factor is an interpolated free group factor due to Dykema's abstract description of the amalgamated free product of multimatrix von Neumann algebras \cite[Theorem 5.1]{Dykema1995Amalgamated}. As $H$ and $K$ are compact, their von Neumann algebras are precisely multimatrix algebras, by the Peter-Weyl theorem, and therefore 
        \[\L(\bmp{F})=\L(H\ast_K H)\simeq \L(H)\ast_{\L(K)}\L(H)=M_0\oplus M_1\]
        where $M_0$ is a finite amenable von Neumann algebra and $M_1$ is a direct sum of interpolated free group factors. The factor $p_K\L(\bmp{F})p_K$ being non-amenable, it must be a corner of an interpolated free group factor, \ie an interpolated group factor itself.
    \end{proof}

    \begin{Ex} For $n\geq 3$ and $m\geq 2$, let $C_n$ be the cyclic group with $n$ elements and $\tilde{F}$ a transitive subgroup of $\Sym_m$.  Let $F=C_n\times \tilde{F}\subset \Sym(\{1,\dots,n\}\times \{1,\dots,m\})$. Then $\bmp{F}$ is non-discrete and satisfies the hypothesis of \Cref{cor:Hecke Burger Mozes}. If $K$ denotes the stabilizer of any edge, then $p_KL(\bmp{F})p_K$ is an interpolated free group factor.
    \end{Ex}

    \subsection{Application 2: HNN Extensions of Locally Compact Groups}\label{subsection:HNN extensions}\label{subsection: Factoriality of HNN extensions}
    In this section, we reformulate the ideas of \cite[Proof for Theorem G]{Raum2021Erratum} in the von Neumann algebraic setting rather than the $\Cstar$-algebraic one, using \Cref{cor:generalizedSuzukicriterion}. As this result offers little novelty compared to \cite{Suzuki2017Elementary}, we do not claim originality for the following results, but we feel it useful to write them down in the von Neumann algebraic setting.

     Let $H$ be a locally compact group, and let $K_{1}$ and $K_{-1}$ be two open subgroups of $H$, and $\theta:K_{1}\to K_{-1}$ be a continuous and open isomorphism of locally compact groups. Then the group \[ \HNN(H,K_{1},K_{-1},\theta)\doteq \left \langle H,t\mid tkt^{-1}=\theta(k)\quad\forall k\in K_{1}\right \rangle\]
    is the \emph{$\HNN$ extension corresponding to $H,K_{1},K_{-1},\theta$}. It admits a unique topology such that the inclusion $H\hookrightarrow \HNN(H,K_{1},K_{-1},\theta)$ is open, and this topology turns the $\HNN$ extension into a locally compact group (see \cite[Proposition 8.B.10]{CornulierHarpe2016MetricGeometry}).  

    Let $G$ be such an $\HNN$ extension. Any element $g\in G$ can be written as a product $g=g_1t^{\epsilon_1}\cdots g_nt^{\epsilon_n}g_{n+1}$ where each $g_k$ belongs to $H$ and $\epsilon_k$ to $\{1,-1\}$. Such an expression is called a \emph{reduced form} if, whenever $\epsilon_{k-1}=-\epsilon_{k}$, the group element $g_k$ does not belong to $K_{\epsilon_k}$. Every $g\in G$ admits such reduced forms, and the quantities $\sigma(g)\doteq \sum_{i=1}^n\epsilon_i$ and $\tau(g)\doteq \sum_{i=1}^n\abs{\epsilon_i}$ are well-defined, that is they do not depend on the choice of a reduced form \cite[Lemma 2.3]{LyndonSchupp1977CombinatorialGroupTheory}.

    \begin{Prop}\label{prop:factoriality result for HNN extensions}
    Let $G=\HNN(H,K_{1},K_{-1},\theta)$ be a locally compact $\HNN$ extension. For $m\in \nn$, let \( K_m\doteq\{ x\in G\mid \forall l\in \{-m,\dots,m\}, t^lxt^{-l}\in H\}\) and denote by $\Z_H(K_{\pm 1})$ the centralizer of $K_{\pm 1}$ in $H$. Assume that $(K_m)_{m\geq 1}$ forms a neighborhood basis at the identity consisting of compact open subgroups, and that there exist $x_{1}\in (\Z_H(K_{1})\setminus K_{1})\cap \ker \Delta_H$ and $x_{-1}\in (\Z_H(K_{-1})\setminus K_{-1})\cap \ker \Delta_H$. Then $\L(L)'\cap \L(G)=\cc$, where $L=\overline{\langle x_{1},x_{-1},t \rangle}\cap \ker \Delta_G$. 
    \end{Prop}
    \begin{proof}
    Let $m\geq 1$. Consider: 
     \[L_m\doteq\overline{\left\langle  g_1t^{\epsilon_1}\dots g_nt^{\epsilon_n}g_{n+1} \text{ reduced form }\middle|\; g_i\in \{x_{-1},x_{1},e\}, \sum_{i=1}^n\epsilon_i=0 \text{ and } \forall l\leq n,\; \sum_{i=n-l}^{n}\epsilon_i\leq m\right\rangle}.\]
    Then $L_m$ is clearly a subgroup of $L$, and $L_{m}$ centralizes $K_m$ (therefore $L_{m}\leq\nmlz{L}{K_m}$). We aim to show that $\ccc{L_m}{gK_m}=\infty$ whenever $g\notin K_m$, which implies the strong form of $\etoile{L}$ used in \Cref{cor:generalizedSuzukicriterion}. Let $g\in G\setminus K_m$. Consider a reduced form of $g$, and distinguish the following cases:
    \begin{description}
        \item[{\boldmath\bfseries If $g\in H\setminus K_m$.}] Let $s\in \{-m,\dots,m\}$ such that $t^sgt^{-s}\notin H$. If $s>0$, then $t^{s-1}gt^{-(s-1)}\notin K_1$. Set $x=x_{-1}t^{-s}x_{1}t^{s}$. Then for any $N\in \nn$, $x^N\in L_{m}$ and $\tau(x^N)=2Ns$. Moreover, 
        \[x^Ngx^{-N}=x^{N-1}x_{-1}t^{-s}x_{1}t(t^{s-1}gt^{-(s-1)})t^{-1}x_{1}^{-1}t^{s}x_{-1}^{-1}x^{N-1}\]
        is a reduced form. Therefore, for any $h$ in $x^N(gK_m)x^{-N}=x^Ngx^{-N}K_m$, $\tau(h)=2(s+1)+2(N-1)\tau(x)$, so that the cosets $\{x^Ngx^{-N}K_m\}_{N\in \nn}$ are mutually disjoint. It implies that $\ccc{L}{gK_m}=\infty$. If $s<0$ the same occurs with $x=x_{1}t^{-s}x_{-1}t^{s}$.
        \item[{\boldmath\bfseries If $g=g_1t^{\epsilon_1}\cdots g_nt^{\epsilon_n}g_{n+1}$ with $n\geq 1$ and $\epsilon_1=-\epsilon_n$.}] With $x=x_{-\epsilon_1}t^{-\epsilon_1}x_{\epsilon_1}t^{\epsilon_1}$ similar arguments imply that $\{x^Ngx^{-N}K_m\}_{N\in \nn}$ are mutually disjoint and therefore $\ccc{L}{gK_m}=\infty$.
        \item[ {\boldmath\bfseries If $g=g_1t^{\epsilon_1}\cdots g_nt^{\epsilon_n}g_{n+1}$ with $n\geq 1$ and $\epsilon_1=\epsilon_n$.}] If $g_1\in K_{\epsilon_1}$, then set $x=t^{\epsilon_1}x_{-\epsilon_1}t^{-\epsilon_1}x_{\epsilon_1}$, and if $g_1\notin K_{\epsilon_1}$, set $x=x_{\epsilon_1}t^{\epsilon_1}x_{-\epsilon_1}t^{-\epsilon_1}$. In both cases, $\{x^Ngx^{-N}K_m\}_{N\in \nn}$ are mutually disjoint and therefore $\ccc{L}{gK_m}=\infty$.
    \end{description}
    \end{proof}
    This method applies in particular to the profinite completions of Baumslag-Solitar groups, whose $\Cstar$-simplicity was proved by Raum in \cite{Raum2021Erratum}, the factoriality of their von Neumann algebra being implicit.
    \begin{Ex}[\cite{Raum2019Csimplicity,Raum2021Erratum}] Let $2\leq \abs{m}<n$ be two integers. Then the \emph{Baumslag-Solitar} group $\operatorname{BS}(m,n)$ is the discrete $\operatorname{HNN}$ extension $\HNN(\langle a\rangle ,\langle a^m\rangle,\langle a^n\rangle,\theta)$ where $\theta(a^m)=a^n$. By Bass-Serre theory, such a group acts on a $(\abs{m}+n)$-regular tree. Let $G(m,n)\doteq \overline{\operatorname{BS}(m,n)}\leq\Aut(\T_{\abs{m}+n})$ be the closure of such a Baumslag-Solitar group in the automorphism group topology. Then $G(m,n)$ is the $\operatorname{HNN}$ extension corresponding to $\overline{\langle a \rangle},\overline{\langle a^m \rangle},\overline{\langle a^n \rangle},\overline{\theta}$, where $\overline{\theta}$ is the natural extension of $\theta$. Therefore, $G(m,n)$ satisfies the assumptions of the previous proposition, as $K_l=\overline{\langle a^{m^ln^l}\rangle}$ effectively defines a basis at the identity in compact open subgroups. By \Cref{prop:factoriality result for HNN extensions}, $L(G(m,n))$ is a factor, of type $\III_{\abs{\frac{m}{n}}}$ by \Cref{cor:non typeI factors}. These factors are non-amenable, by \cite[Theorem 9.2]{Raum2019Csimplicity}.
    \end{Ex}

    \subsection{Application 3: Neretin Groups}\label{subsection:Neretin}

    Let $d,k\in \nn^*$, and denote by $\Tdk$ the rooted tree such that the root $v_0$ has degree $k$, and the other vertices have degree $d$. For the rest of this section, we fix an embedding of this tree in the plane that induces an order on the descendants of any vertex of $\Tdk$. Denote by $h:\V(\Tdk)\to \nn$ the \emph{height} function that assigns its distance from the root to each vertex. Denote by $\V^{(n)}(\Tdk)$ the set of vertices of height $n$. Given $v\in \V(\Tdk)$, we denote by $\Tdk^v$ the subtree of $\Tdk$ consisting of descendants of $v$, and by $\Tdk^{\leq n}$ the finite subtree of $\Tdk$ obtained by keeping only vertices (and edges between them) at distance at most $n$ from the root. If $v\neq v_0$, then $\Tdk^v$ is isomorphic to $\Tdd$, the isomorphism being canonical once an embedding of each tree is fixed.
    
    The boundary $\bTdk$ of the tree is the set of infinite geodesics (viewed as non-backtracking sequences of adjacent vertices) rooted in $v_0$. Given $\xi,\xi'\in \bTdk$, let $h(\xi,\xi')\doteq \sup \{h(v)\mid v\in \xi\cap\xi'\}$. We define the visual distance on $\bTdk$ by setting $\delta(\xi,\xi')=d^{-h(\xi,\xi')}$. With this distance, $\bTdk$ is a Cantor set of unit diameter. Given $v\in \V(\Tdk)$, let $B^v\subset \bTdk$ be the set of $\xi\in \bTdk$ such that $v\in \xi$. Note that $B^v$ is the clopen ball of $\bTdk$ consisting of elements at distance at most $d^{-h(v)}$ of any of its elements. Denote by $P^{(n)}=\{B^v\mid v\in \V^{(n)}(\Tdk)\}$ the recovering of $\bTdk$ in such balls of radius $d^{-n}$.

    We say that an element $\phi\in \Homeo(\bTdk)$ is a \emph{local similarity} if there exists two partitions $\bTdk=\bigsqcup_{i=1}^nB^{v_i}$, $\bTdk=\bigsqcup_{i=1}^nB^{w_i}$, such that for every $i\in\{1,\dots, n\}$, $\phi|_{ B^{v_i}}$ is a similarity onto $B^{w_i}$ (in symbols, for every $\xi,\xi'\in B^{v_i}$, $\phi(\xi)$ and $\phi(\xi')$ both belong to $B^{w_i}$ and $\delta(\phi(\xi),\phi(\xi'))=d^{h(v_i)-h(w_i)}\delta(\xi, \xi')$). A local similarity is a \emph{local isometry} if $h(v_i)=h(w_i)$ for every $i\in \{1,\dots,n\}$.

    \begin{Def}[{\textmd{\cite{Neretin1992Combinatorial}}}]
        The \emph{Neretin group} $\Ndk$ is the group of local similarities of $\bTdk$. 
    \end{Def}

    Given $G\leq\Homeo(\bTdk)$ a group of homeomorphisms of the boundary and $A\subset \bTdk$, we consider the following subgroups of $G$:
    \begin{align*}
        \Fix_G(A)&\doteq \{g\in G\mid \forall x\in A, gx=x\}\\
        \Stab_G(A)&\doteq \{g\in G\mid \forall x\in A, gx\in A\}\\
        \Isom_G(A)&\doteq \{g\in G\mid g|_{ A} \text{ is an isometry }\}
    \end{align*}
    We introduce the following subgroups of the Neretin group: 
    \begin{align*}
        \Odk&\doteq\{g\in \Ndk\mid g\text{ is a local isometry}\}\\
        \Kdk&\doteq\Isom_{\Ndk}(\bTdk)\\
        \Kdkn{n}&\doteq\bigcap_{B\in P^{(n)}}\Stab_{\Kdk}(B)\\
        \Odkn{n}&\doteq\bigcap_{B\in P^{(n)}}\Isom_{\Odk}(B)
    \end{align*}
    Note that $\Odkn{n}$ is the normalizer of $\Kdkn{n}$ in $\Ndk$, and that $\ind{\Odkn{n}}{\Kdkn{n}}=\abs{\V^{(n)}(\Tdk)}!$.

    An equivalent description of the Neretin group consists in realizing local similarities as equivalence classes of almost automorphisms of the tree itself, an almost automorphism being a triplet $(\phi, \T_1, \T_2)$ where $\T_1$ and $\T_2$ are complete finite subtrees of $\Tdk$, and $\phi:\Tdk\setminus \T_1\to \Tdk\setminus \T_2$ is an isomorphism of forest. Under this identification, it is clear that $\Kdk$ corresponds to $\Fix_{\Aut(\Tdk)}(v_0)$ and therefore is endowed with its natural profinite topology coming from its rooted action on $\Tdk$. For the moment, it is the only group to be topologized. However, $\Kdk$ is a commensurated subgroup of $\Ndk$; in that respect, there exists a unique locally compact totally disconnected topology on $\Ndk$ such that the inclusion $\Kdk\hookrightarrow \Ndk$ is continuous and open (\cite[Lemme 4.3]{LeBoudec2017CompactPresentability}). For this topology, $\Odkn{n}$ is compact and open (as $\Kdkn{n}$ is of finite index in $\Odkn{n}$), and $\Odk$ is unimodular and amenable because $\Odk=\bigcup_{n\geq 0}\Odkn{n}$.

    \begin{The}[\Cref{mainthm:Neretin}]\label{thm:Neretin} The regular representation $\lambda_{\Ndk}$ of a Neretin group $\Ndk$ is factorial. More precisely, 
    \[\L(\Odk)'\cap\L(\Ndk)=\cc.\]
    As a consequence, $\L(\Odk)$ is the hyperfinite $\II_{\infty}$ factor and $\L(\Ndk)$ is a type $\II_{\infty}$ factor.
    \end{The}
    \begin{proof}
        Let $\mu$ be a left Haar measure on $\Ndk$, normalized so that $\mu(\Kdk)=1$. Recall that $(\Kdkn{n})_{n\geq 0}$ is a basis at the identity of $\Ndk$ in compact open subgroups. We will show that $\Ndk$ has property $\etoile{\Odk}$ with respect to this basis. To this end, we introduce the following notion: given a subgroup $G\leq\Homeo(\bTdk)$ and $A\subset \bTdk$, the \emph{rigid stabilizer} of $A$ in $G$ is the following subgroup: 
        \[\rist_G(A)\doteq \Fix_{G}(\bTdk\setminus A).\]

        Let $g\in \Ndk\setminus \{e\}$. There exists a ball $B^{w}\subset \bTdk$ such that $g(B^{w})\neq B^{w}$, and such that $g|_{ B^{w}}$ is a similarity with ratio $\alpha >0$. Up to taking a smaller ball inside $B^{w}$, we may assume that $g(B^{w})\cap B^{w}=\emptyset$. Let $n_0=h({w})$. Define for $n\geq n_0$ the following subgroups of $\Odkn{n}$, where the isomorphisms are obtained by identifying $B^{w}$ and $\Tdd$: 
        \begin{align*}
            H^{(w,n)}&=\rist_{\Odkn{n}}(B^{w})\simeq \O_{d,d}^{(n-n_0)}\\
            L^{w}&=\rist_{\Kdk}(B^{w})\simeq K_{d,d}
        \end{align*}
        The group $H^{(w,n)}$ is the \emph{rigid stabilizer} of $B^{w}$ in $\Odkn{n}$, and $L^{w}$ is the rigid stabilizer of $B^{w}$ in $\Kdk$.
        Let $h\in H^{(w,n)}$ such that $h\notin L^{w}$ and pick $\xi_1,\xi_2\in B^{w}$ such that the distance between $\xi_1$ and $\xi_2$ is not preserved by $h^{-1}$. Then
        \begin{align*}
        &\delta(hgh^{-1}\xi_1,hgh^{-1}\xi_2)
        =\delta(gh^{-1}\xi_1,gh^{-1}\xi_2)\\
        &=\alpha \delta(h^{-1}\xi_1,h^{-1}\xi_2)\\
        &\neq \alpha \delta(\xi_1,\xi_2). 
        \end{align*}
        Therefore, $hgh^{-1}\Kdkn{n}\cap g\Kdkn{n}=\emptyset$ so that $\ccc{\Odk}{g\Kdkn{n}}\geq \ind{H^{(w,n)}}{L^{w}}$.

        Note that $\mu(\Kdkn{n})=\ind{\Kdk}{\Kdkn{n}}^{-1}$: we therefore establish the estimates necessary to ensure that $\left(\ind{H^{(w,n)}}{L^{w}}\ind{\Kdk}{\Kdkn{n}}^{-2}\right)_{n\geq n_0}$ tends to infinity. 

        \begin{itemize}
            \item Recall that $H^{(w,n)}/L^{w}$ is isomorphic to $\O_{d,d}^{(n-n_0)}/K_{d,d}$. Therefore, 
            \[\ind{H^{(w,n)}}{L^{w}}=\ind{\O_{d,d}^{(n-n_0)}}{K_{d,d}}=\frac{\ind{\O_{d,d}^{(n-n_0)}}{K_{d,d}^{(n-n_0)}}}{\ind{K_{d,d}}{K_{d,d}^{(n-n_0)}}}.\]
            \item $\Odkn{n}/\Kdkn{n}$ is the group of permutations of $\V^{(n)}(\Tdk)$. Therefore, $\ind{\Odkn{n}}{\Kdkn{n}}=\abs{V^{(n)}(\Tdk)}!$, whence 
            \[ \ind{\O_{d,d}^{(n-n_0)}}{K_{d,d}^{(n-n_0)}} = \abs{\V^{(n-n_0)}(\T_{d,d})}! = (d^{n-n_0})!\geq \left(\frac{d^{n-n_0}}{e}\right)^{d^{n-n_0}}.\]
            \item $\Kdk/\Kdkn{n}$ acts faithfully on the finite tree $\Tdkn$. Then
            \[\ind{\Kdk}{\Kdkn{n}}\leq \abs{\Aut(\Tdkn)}=k!\prod_{i=1}^{n-1}(d!)^{\abs{\V^{(i)}(\Tdk)}}=k!(d!)^{k\sum_{i=0}^{n-2}d^i} \leq k!(d!)^{kd^n}.\]
        \end{itemize}
        Finally, for $n\geq n_0$,
        \begin{align*}
            \ccc{\Odk}{g\Kdkn{n}}\mu(\Kdkn{n})^2
            \geq \frac{\ind{\O_{d,d}^{(n-n_0)}}{K_{d,d}^{(n-n_0)}}}{\ind{K_{d,d}}{K_{d,d}^{(n-n_0)}}}\abs{\Kdk/\Kdkn{n}}^{-2}
            \geq \frac{\left(\frac{d^{{n-n_0}}}{e}\right)^{d^{n-n_0}}}{(d!)^{d^{n-n_0}}}\frac{1}{k!(d!)^{2kd^n}}\tend{n}{\infty}\infty.
        \end{align*}
    Therefore, given any nontrivial group element $g$, $\ccc{\Odk}{g\Kdkn{n}}\mu(\Kdkn{n})^2\to \infty$, and by \Cref{the:irreducible inclusion for groups with etoile_H}, $\L(\Odk)'\cap \L(\Ndk)=\cc$. By \Cref{cor:non typeI factors}, $\L(\Odk)$ and $\L(\Ndk)$ are type $\II_{\infty}$ factors. As $\Odk$ is amenable, $\L(\Odk)$ is the hyperfinite $\II_{\infty}$ factor.
    \end{proof}

\begin{Rem}
    \begin{enumerate}
        \item\label{itm:genNer2} The proof of \Cref{thm:Neretin} carries over verbatim to the case of the natural generalization of the Neretin groups introduced in \cite[Section 7]{LeBoudec2017CompactPresentability} under the name of \emph{almost automorphism groups associated with closed regular branch groups}.
        \item\label{itm:genNer3} Another natural generalization of Neretin groups is provided by the notion of \emph{topological full group} associated with an action $G\acts \partial\T$ of a profinite branch group, defined in \cite[Section 2, Section 6.2]{GarridoReid2025LocallyCompactPiecewiseFullGroups}. Likewise, the proof of \Cref{thm:Neretin} applies to these locally compact groups, therefore providing other examples of totally disconnected locally compact groups with factorial regular representations. 
        \item Let $F\leq \Sym_d$, such that $F\acts\{1,\dots,d\}$ is non-free. The \emph{coloured Neretin group} $\N_F$ associated to $F$, defined in \cite{Lederle2019ColouredNeretin}, is the piecewise full group associated to the action of the Burger-Mozes group $U(F)$ on the boundary of $\T_d$. When $F$ is transitive, $\N_F$ is one of the groups defined in items \ref{itm:genNer2} and \ref{itm:genNer3}, and therefore $\L(\N_F)$ is a factor. But the estimates in the proof of \Cref{thm:Neretin} can easily be adapted to show that, whenever the only fixed point of $\Fix_F({i})$ is $\{i\}$ for every $i\in \{1,\dots,d\}$, $\L(\N_F)$ is a factor. Indeed, this technical hypothesis ensures that the size of the smallest orbit of the action of the rigid stabilizer of $B^v$ on $\T_{d,k}^{v,(n)}$ grows at least linearly in $n$. This allows us to recover the asymptotics used to obtain $\etoile{\Odk}$.
        \item While the factor $\L(\Odk)$ is entirely determined by its amenability, we do not know much about $\L(\Ndk)$ itself. It will be interesting to know if $\L(\Ndk)$ is amenable.
    \end{enumerate}
\end{Rem}

\section{Factoriality Criterion for Crossed Products}\label{section:Factoriality Criterion for Crossed Products}
    \subsection{Preliminaries}\label{subsection:crossed products preliminaries}
        \paragraph{Von Neumann algebras. } In this section, we will consider von Neumann algebras as abstract operator algebras.  Consequently, given $M$ a von Neumann algebra, we denote by $(M_*,\norm{\cdot}_*)$ its predual, by $\norm{\cdot}_M$ the dual norm on $M$, and the \emph{weak} topology on $M$ is the weak-$*$ topology associated to the duality $M=(M_{*})^*$. Given a positive state $\psi\in M^+_*$ and $x\in M$ we define $\norm{x}_{\psi}\doteq \psi(x^*x)^{\frac{1}{2}}$. The strong topology on $M$ is the topology generated by the seminorms $\{\norm{\cdot}_{\psi}\mid \psi\in M_*^+\}$. We write the representations explicitly when we want our von Neumann algebras to act on Hilbert spaces.

        \paragraph{Standard form. } Let $M$ be a von Neumann algebra, and denote by $M^{\op}$ its opposite von Neumann algebra, that is the von Neumann algebra with the same elements and adjoint operation as in $M$, but with the reversed product $x\cdot_{\op}y\doteq yx$. Following \cite{Haagerup1975StandardForm}, we denote by $(M,\Ldeux M, J,P)$ the \emph{standard form} of $M$, and by $\lambda_M:M\to \B(\Ldeux M)$ the associated left regular representation. We define its right regular representation $\rho_M:M^{\op}\to \B(\Ldeux M)$ by $\rho_M(x)=J\lambda_M(x^*) J$, so that $J\lambda_M(M)J=\rho_M(M^{\op})=\lambda_M(M)'$. For example, for $G$ a locally compact group, the standard form  of $\L(G)$ is given by $(\L(G),\Ldeux(G,\mu_G),J,P)$, where 
        \[J:\Ldeux G\ni\xi\mapsto\left(h\mapsto \Delta_G(h^{-1})^{\frac{1}{2}}\overline{\xi(h^{-1})}\right)\in \Ldeux G,\] 
        $\lambda_{\L(G)}(u_g)=\lambda_g$, $\rho_{\L(G)}(u_g)=\rho_{g^{-1}}$ and $P=\overline{\{ \xi\ast (J\xi)\mid \xi \in C_c(G)\}}$, where $C_c(G)$ denotes the set of continuous functions with compact support on $G$.

        \paragraph{Automorphisms groups, actions, and unitary implementations. } The \emph{automorphism group of $M$}, denoted by $\Aut(M)$, acts on $M_*$ by $\theta(\omega)=\omega\circ\theta$ for $\theta\in\Aut(M)$ and $\omega\in M_*$. It is endowed with the \emph{$u$-topology} which is the topology of pointwise norm convergence on $M_*$: a net $(\theta_i)_{i\in I}$ converges to $\theta$ if and only if for all $\omega\in M_*$, $\norm{\theta_i(\omega)-\theta(\omega)}\to 0$.
        Given $\theta\in \Aut(M)$ there exists a unique unitary $v_{\theta}\in \B(\Ldeux M)$ such that for all $x\in M$, $\theta(x)=v_{\theta} xv_{\theta}^*$: it is called the \emph{unitary implementation} of $\theta$. Given $G$ a locally compact group, a continuous action $\alpha:G\acts M$ is a continuous morphism $\alpha:G\to \Aut(M)$ for the $u$-topology on $\Aut(M)$. Once the action is fixed, we denote by $v_g\in \U(\Ldeux M)$ the unitary implementation of $\alpha_g$.

        \paragraph{Smoothing map.} Let $\alpha:G\to M$ be a continuous action of a locally compact group on a von Neumann algebra and $K$ be a compact open subgroup of $G$. By continuity of the action and given $\omega\in M_*$, the application $K\ni g\mapsto \alpha_g(\omega)\in (M_*,\norm{\cdot}_*)$ is continuous and therefore Bochner integrable. We introduce
        \[ \widetilde{\alpha_K}:M_*\ni \omega \mapsto \frac{1}{\mu_G(K)}\int_K\alpha_k(\omega)d\mu_G(k)\in (M_*,\norm{\cdot}_*)\]
        the \emph{smoothing map on $M_*$}, and its dual \emph{smoothing map on $M$} 
        \[ \alpha_K:M\ni x\mapsto (\omega\mapsto \widetilde{\alpha_K}(\omega)(x))\in (M_*)^*=M.\]
        By construction, $\alpha_K$ is weakly continuous.

        \paragraph{Crossed products.} Given an action $\alpha:G\acts M$, we define the following representations of $G$ on $\H\doteq\Ldeux M\otimes\Ldeux G\simeq \Ldeux (G;\Ldeux M)$: 
        \begin{align*}
             \1\otimes\lambda_G: G &\to \B(\H),& (\1\otimes\lambda_G)(g)\xi(s)&=\xi(g^{-1}s) &\xi&\in \H,g\in G, s\in G,\\
             \rho_G^{\alpha}:G &\to \B(\H),& \rho_G^{\alpha}(g)\xi(s)&=v_g\xi(sg) &\xi&\in \H,g\in G, s\in G.
        \intertext{We also define the following representations of $M$ and $M^{\op}$ on the same space:}
            \lambda_M^{\alpha}:M&\to \B(\H),& \lambda_M^{\alpha}(a)\xi(s)&=\alpha_{s^{-1}}(a)\xi(s)&\xi&\in \H,a\in M, s\in G, \\
            \rho_M\otimes \1:M^{\op}&\to \B(\H),& (\rho_M\otimes \1)(a)\xi(s)&=\rho_M(a)\xi(s)&\xi&\in \H,a\in M^{\op}, s\in G.
        \end{align*}

        The \emph{crossed product} associated to the action $\alpha:G\acts M$ is the von Neumann algebra generated by a copy of $M$ and a family of unitaries $\{u_g\mid g\in G\}$, verifying $u_gxu_g^*=\alpha_g(x)$ for $g\in G$ and $x\in M$, that admits a faithful representation $\lambda_{M\rtimes G}:M\rtimes_{\alpha} G\to \B(\Ldeux M \pt \Ldeux G)$ such that $\lambda_{M\rtimes G}(u_g)=(\1\otimes\lambda_G)(g)$ and $\lambda_{M\rtimes G}(x)=\lambda_M^{\alpha}(x)$. This notation is coherent, as $(\H,\lambda_{M\rtimes G})$ is indeed a standard representation of $M\rtimes_{\alpha}G$. By \cite[X.1.21]{Takesaki2003TheoryOperator},
        \begin{equation*}\label{eq:commutation relation for crossed products}\tag{1}
         \left((\1\otimes\lambda_G)(G)\cup \lambda_M^{\alpha}(M)\right)'=\left(\rho_G^{\alpha}(G)\cup (\rho_M\otimes \1)(M^{\op})\right)''.
        \end{equation*}

    \subsection{Fourier Coefficients on Crossed Products}\label{subsection:crossed products Fourier Coeffs}
        In this section, we adapt the notion of Fourier coefficient on group von Neumann algebras to crossed products by locally compact groups. For any Borel subset $A\subset G$ with non-zero finite Haar measure, $\xi_A\doteq \1_A/{\mu_G(A)}^{1/2}$ still denote the normalized indicator function of $A$ in $\Ldeux(G,\mu_G)$. Let $\alpha:G\acts M$ be a continuous action of a locally compact group on a von Neumann algebra. 

        \begin{Def}\label{def:fourier coefficients for crossed products}
             Let $x\in M\rtimes_{\alpha} G$, $K$ be a compact open subgroup of $G$ and $g\in G$.  We define the \emph{Fourier coefficient of $x$ at the coset $gK$} as the unique element $\fc{x}{gK}$ of $M$ such that given $\xi, \eta \in \Ldeux(M)$, 
        \[ \ps{\lambda_M(\fc{x}{gK})\xi}{\eta}=\ps{\lambda_{M\rtimes G}(x)\xi\otimes \xi_K}{(\1\otimes\lambda_G)(g)\eta\otimes \xi_K}.\]
        We therefore consider $\fc{x}{}$ as a function from  $\lcocosets{G}$ to $M$.
    \end{Def}
    \begin{proof}[Proof that $\fc{x}{gK}$ is well defined.] It is clear that $\fc{x}{gK}$ can be defined as an operator in $\B(\Ldeux M)$. However, we have to check that $\fc{x}{gK}$ indeed belongs to $\lambda_M(M)$ and can therefore be defined as an element of $M$ itself. Given $b\in M$, 
            \begin{align*}
                \ps{\fc{x}{gK}\rho_M(b)\xi }{\eta}
                &=\ps{\lambda_{M\rtimes G}(x)(\rho_M\otimes 1)(b)\xi\otimes \xi_K}{(\1\otimes\lambda_G)(g)\eta\otimes \xi_K}\\
                &=\ps{(\rho_M\otimes 1)(b)\lambda_{M\rtimes G}(x)\xi\otimes \xi_K}{(\1\otimes\lambda_G)(g)\eta\otimes \xi_K}\\
                &=\ps{\lambda_{M\rtimes G}(x)\xi \otimes \xi_K}{(\rho_M\otimes 1)(b^*)(\1\otimes\lambda_G)(g)\eta\otimes \xi_K}\\
                &=\ps{\fc{x}{gK}\xi}{\rho_M(b^*)\eta}=\ps{\rho_M(b)\fc{x}{gK}\xi}{\eta}.
            \end{align*}
            Therefore, $\fc{x}{gK}$ commutes with $\rho_M(M)$, and can be identified with the element of $M$ characterized by
             \[ \ps{\lambda_M(\fc{x}{gK})\xi}{\eta}=\ps{\lambda_{M\rtimes G}(x)\xi\otimes \xi_K}{(\1\pt\lambda_G)(g)\eta\pt\xi_K}.\]
        \end{proof}

        We shall now establish the crossed-product versions of \Cref{lem:continuity properties of Fourier coefficients for groups}, \cref{lem:monotonicity of phi/mu} and \cref{lem:caracterization of scalar elements by Fourier coefficients}.

    \begin{Prop}\label{lem:continuity properties of Fourier coefficients for groups} Let $x\in M\rtimes_{\alpha} G$, $K$ be a compact open subgroup of $G$ and $g\in G$. Then 
        \begin{enumerate}
            \item \label{itm:fc.cp.1}$\fc{x}{gK}$ is $K$-invariant. 
            \item \label{itm:fc.cp.2} Let $a\in M$. Then $\fc{a}{K}=\alpha_{K}(a)$, and if $g\notin K$, then $\fc{a}{gK}=0$.
            \item \label{itm:fc.cp.3}$\fc{x}{}:\lcocosets{G}\ni gK\mapsto \fc{x}{gK}\in (M, \norm{\cdot}_M)$ is bounded by $\norm{x}_{M\rtimes G}$.
            \item \label{itm:fc.cp.4}$M\rtimes_{\alpha}G\ni x\mapsto \fc{x}{gK} \in M$ is linear and weakly continuous.
        \end{enumerate}
    \end{Prop}
     \begin{proof}
        \begin{enumerate}
            \item The Fourier coefficient $\fc{x}{gK}$ is $K$-invariant as given $k\in K$, 
            \begin{align*}
                \ps{\lambda_M(\alpha_{k^{-1}}(\fc{x}{gK}))\xi}{\eta}
                &=\ps{v_k^*\lambda_M(\fc{x}{gK})}v_k\xi{\eta}
                =\ps{\lambda_{M\rtimes G}(x)v_k\xi\pt\xi_K}{(\1\pt\lambda_G)(g)v_k\eta\pt \xi_{K}}\\
                &=\ps{\lambda_{M\rtimes G}(x)\rho_G^{\alpha}(k)\xi\pt\xi_K}{(\1\pt\lambda_G)(g)\rho_G^{\alpha}(k)\eta\pt \xi_{K}}\\
                &=\ps{\lambda_{M\rtimes G}(x)\xi\pt\xi_K}{(\1\pt\lambda_G)(g)\eta\pt \xi_{K}}
                =\ps{\lambda_M(\fc{x}{gK})\xi}{\eta}.
            \end{align*}
            \item  If $a\in M$, then 
    \begin{align*}
        \ps{\lambda_M(\fc{a}{gK})\xi}{\eta}
        &=\ps{\lambda_{M\rtimes G}(a)(\xi\pt\xi_K)}{(\1\pt\lambda_G)(g)\eta\pt\xi_K}
        =\ps{\lambda_{M}^{\alpha}(a)(\xi\pt\xi_K)}{\eta\pt \xi_{gK}}\\
        &=\frac{1}{\mu_G(K)}\int_{K\cap gK}\ps{\lambda_M(\alpha_{k}^{-1}(a))\xi}{\eta}d\mu_G(k)\\
        &=\begin{cases}
        \ps{\lambda_M(\alpha_K(a))\xi}{\eta} & \text{if } g\in K, \\
        0 & \text{otherwise}
    \end{cases}.
    \end{align*}
    \end{enumerate}
    The items \ref{itm:fc.cp.3} and \ref{itm:fc.cp.4} are clear from the definitions.
    \end{proof}

    \begin{Prop}\label{lem:monotonicity of phi/mu for crossed products}
        Let $x\in M\rtimes_{\alpha}G$, $g\in G$ and $K, K'$ be compact open subgroups of $G$ such that $K'\leq K$. Pick $g_1,\dots, g_n$ in $gK$ such that $gK=\bigsqcup_{i=1}^ng_iK'$. Then  
        \[ \fc{x}{gK}=\alpha_{K}\left(\sum_{i=1}^n\fc{x}{g_iK'}\right).\]
    \end{Prop}
    \begin{proof}
            Given $a\in M$ and $\gamma\in G$,
            \begin{align*}
                \alpha_{K}\left(\sum_{i=1}^n\fc{u_{\gamma}a}{g_iK'}\right)
                &=\alpha_{K}\left(\sum_{i=1}^n\fc{u_{\gamma^{-1}g_i}^*a}{K'}\right)
                =\alpha_{K}\left(\sum_{i\mid \gamma\in g_iK'}\alpha_{K'}(a)\right)
                =\begin{cases}
                0 & \text{if } \gamma\notin gK, \\
                \alpha_K(a) & \text{if } \gamma \in gK
                \end{cases}\\
                &=\fc{u_{\gamma}a}{gK}
            \end{align*}
            By linearity of $x\mapsto \fc{x}{K}$ and $x\mapsto \alpha_K(x)$, by weak density of $\vect\{u_{\gamma} a \mid \gamma \in G, a\in M\}$ in $M\rtimes_{\alpha} G$ and by weak continuity of $x\mapsto \alpha_K(x)$ and $x\mapsto \fc{x}{K}$, we obtain the equality for every $x\in M\rtimes_{\alpha} G$.
    \end{proof}

    \begin{Prop}\label{lem:condition for nullity of an element in a crossed product}
        Assume that $G$ is totally disconnected. Then, for any $x\in M\rtimes_{\alpha}G$, 
        \[  \fc{x}{}=0 \text { if and only if } x=0.\]
        As a consequence, $x$ is in $M$ if and only if for every compact open subgroup $K$ of $G$ and for every $g\notin K$, $\fc{x}{gK}=0$. 
    \end{Prop}
    \begin{proof}
    Let $x\in M\rtimes_{\alpha}G$ such that $\forall K\co G, \forall g\in G, \fc{x}{gK}=0$.  Given $\xi,\eta\in \Ldeux M$, $h\in G, K\co G$ we have that $\ps{\lambda_{M\rtimes G}(x)\xi\otimes \xi_{hKh^{-1}}}{(1\pt\lambda_G)(g)\eta\otimes \xi_{hKh^{-1}}}=0$ for every $g\in G$, and therefore $\lambda_{M\rtimes G}(x)(\Ldeux M \otimes \xi_{hKh^{-1}})\subset (\Ldeux M \otimes \xi_{g{hKh^{-1}}})^{\perp}$. Recall that, for $h\in G$, $\rho^{\alpha}_G(h)$ commutes with $\lambda_{M\rtimes G}(x)$ and with $(\1\pt\lambda_G)(g)$. Then $\lambda_{M\rtimes G}(x)(\Ldeux M\otimes \xi_{hK})\subset (\Ldeux M\otimes \xi_{g{h^{-1}}K})^{\perp}$ for every $g\in G$, as $\rho^{\alpha}_G(h)(\Ldeux M \otimes \xi_K)=\Ldeux M\otimes \xi_{Kh}$. Therefore for any $K\co G$, $\lambda_{M\rtimes G}(x)(\Ldeux M\otimes \Ldeux(G/K))\subset (\Ldeux M\otimes \Ldeux (G/K))^{\perp}$. 

    Let $K'\co K$: $\Ldeux M\otimes \Ldeux(G/K)\subset \Ldeux M\otimes \Ldeux(G/K')$ and thus $\lambda_{M\rtimes G}(x)(\Ldeux M\otimes \Ldeux(G/K))\subset (\Ldeux M\otimes \Ldeux(G/K'))^{\perp}$. By density of $\bigcup_{K'\co K} \Ldeux M\otimes \Ldeux(G/K')$ in $\Ldeux M \otimes \Ldeux G$, $\lambda_{M\rtimes G}(x)(\Ldeux M\otimes \Ldeux(G/K))=\{0\}$ for every $K\co G$ and by density of $\bigcup_{K\co G} \Ldeux M\otimes \Ldeux(G/K)$ in $\Ldeux M \otimes \Ldeux G$, we obtain $x=0$.

    Now, we prove the consequence. The net $(\fc{x}{K})_{K\co G}$ is bounded by $\norm{x}_{M\rtimes_{\alpha} G}$, and therefore we can choose a cofinal subnet $(K_i)_{i\in I}$ and $a\in M$ such that $\fc{x}{K_i}$ weakly converges to $a$. For every $K\co G$ and for every $g\notin K$, $\fc{x-a}{gK}=0$. By  \Cref{lem:monotonicity of phi/mu for crossed products}, for all $K\co G$ and $j\in I$ such that $K_j\co K$, $\fc{x-a}{K}=\alpha_{K}(\fc{x-a}{K_j})$. The net $\fc{x-a}{K_j}$ weakly converges to $0$ as $j$ goes to infinity and therefore for every $K\co G$, $\fc{x-a}{K}=0$. Therefore, by the first part of the proposition, $x=a\in M$.
    \end{proof}

    \subsection{Factoriality Criterion for Crossed Products}\label{subsection: crossed products factoriality}
    We finally establish the crossed product analog of \Cref{lem:hilbert inequality}, and we prove \Cref{mainthm:factoriality for crossed products}.
    \begin{Lem}\label{lem:hilbert inequality for crossed products}
        Let $\alpha:G\acts (M,\psi)$ be a continuous action of a totally disconnected locally compact group on a von Neumann algebra $M$ endowed with a faithful normal state $\psi$.
        Let $H_{\psi}=\{h\in G\mid \psi\circ \alpha_h=\psi\}$ be the subgroup of $\psi$-preserving elements. Then for any $H\cl H_{\psi}$ and $z\in \L(H)'\cap (M\rtimes_{\alpha}G)$, 
         \begin{flalign*}
            &&\norm{z}_{\infty}^2\geq \ccc{H}{gK}\norm{\fc{z}{gK}}_{\psi}^2 && \text{for every $K\co G$, $g\in G$.}
         \end{flalign*}
    \end{Lem}
    \begin{proof}
        Let $A=gK$, and denote by $\xi_{\psi}\in \Ldeux M$ the unit vector such that for all $x\in M$, $\psi(x)=\ps{x\xi_{\psi}}{\xi_{\psi}}$.
        Let $\eta_{\psi}\doteq \lambda_M(\fc{z}{gK})\xi_{\psi}\in \Ldeux M$. Then $\norm{\fc{z}{gK}}_{\psi}^2=\ps{\lambda_{M\rtimes G}(x)\xi_{\psi}\otimes \xi_K}{\eta_{\psi}\otimes \xi_A}$.
        Given any $B\in\Ccc{H}{gK}$, let $h_B\in \nmlz{H}{K}$ such that $B=h_BAh_B^{-1}$. Then $((\1\pt\lambda_G)(h_B)\rho_G^{\alpha}(h_B)\eta_{\psi}\otimes \xi_{A})_{B\in \Ccc{H}{gK}}$ is a family of orthogonal vectors of norm at most one. Let $z$ be as in the lemma:
   \begin{align*}
       \norm{z}_{\infty}^2&
       \geq \norm{\lambda_{M\rtimes G}(z)\xi_{\psi}\otimes\xi_{K}}_{\Ldeux M\pt \Ldeux G}^2\\
       &\geq \sum_{B\in \Ccc{H}{gK}}\abs{\ps{\lambda_{M\rtimes G}(z)\xi_{\psi}\otimes\xi_{K}}{(\1\pt\lambda_G)(h_B)\rho_G^{\alpha}(h_B)\eta_{\psi}\otimes \xi_{A}}}^2\\
       &= \sum_{B\in \Ccc{H}{gK}}\abs{\ps{\lambda_{M\rtimes G}(z)\rho_G^{\alpha}(h_{B}^{-1})(\1\pt \lambda_G)(h_{B}^{-1})\xi_{\psi}\otimes\xi_{K}}{\eta_{\psi}\otimes \xi_{A}}}^2 \\
       &= \sum_{B\in \Ccc{H}{gK}}\abs{\ps{\lambda_{M\rtimes G}(z)(\xi_{\psi}\otimes\xi_{K})}{\eta_{\psi}\otimes \xi_{A}}}^2=\ccc{H}{gK}\abs{\ps{\lambda_M(\fc{z}{gK})\xi_{\psi}}{\eta_{\psi}}}^2\\
       &= \ccc{H}{gK}\norm{\fc{z}{gK}}_{\psi}^2
   \end{align*}
   \end{proof}

   \begin{The}[\Cref{mainthm:factoriality for crossed products}]\label{the:factoriality for crossed products}
    Let $G$ be a locally compact group with property \eqref{eq:etoileH} relatively to a closed subgroup $H$ and $\alpha:G\acts (M,\psi)$ be a continuous action
    of $G$ on a von Neumann algebra $M$ endowed with a faithful normal state $\psi$. Let $G_{\psi}=\{g\in G\mid \psi\circ \alpha_g=\psi\}$ be the subgroup of $\psi$-preserving elements. If $G_{\psi}$ is open and $H\leq G_{\psi}$, then
    \[\L(H)'\cap (M\rtimes_{\alpha} G)=M^H,\]
    where $M^H$ is the algebra of $H$-invariant elements in $M$.
    In particular, if $H\acts \Z(M)$ is ergodic, then $M\rtimes_{\alpha}G$ is a factor. 
    \end{The}

    \begin{proof}
        We first establish the following inequality: given $x\in M\rtimes_{\alpha}G$, $K\co G$ and $g\in G$, if $\psi\in M_*$ is invariant under $\alpha_{ K}$, then there exists $g'\in gK$ such that 
        \[\frac{\norm{\fc{x}{g'K'}}_{\psi}}{\mu_G(K')}\geq \frac{\norm{\fc{x}{gK}}_{\psi}}{\mu_G(K)}.\]
        Indeed, using \Cref{lem:monotonicity of phi/mu for crossed products} and the $\alpha_{ K}$-invariance of $\psi$, we obtain
            \begin{align*}
                \frac{\norm{\fc{x}{gK}}_{\psi}}{\mu_G(K)}
                =\norm{ \sum_{i=1}^N \frac{\alpha_K(\fc{x}{g_iK'})}{\mu_G(K)}}_{\psi}
                =\frac{1}{N}\norm{ \sum_{i=1}^N \frac{\alpha_K(\fc{x}{g_iK'})}{\mu_G(K')}}_{\psi}
                \leq \frac{1}{N}\sum_{i=1}^N \frac{\norm{\alpha_K(\fc{x}{g_iK'})}_{\psi}}{\mu_G(K')}\\
                \leq \frac{1}{N}\sum_{i=1}^N \frac{\norm{\fc{x}{g_iK'}}_{\psi}}{\mu_G(K')}
                \leq \max \frac{\norm{\fc{x}{g_iK'}}_{\psi}}{\mu_G(K')}.    
            \end{align*}

        Let $z\in \L(H)'\cap (M\rtimes_{\alpha} G)$. Assume by contradiction that there exists $K\co G$ and $g\notin K$ such that $\fc{z}{gK}\neq 0$. 
        By \Cref{lem:monotonicity of phi/mu for crossed products} and using the fact that $H_{\psi}$ is open, we may assume that $K\co H_{\psi}$. Consider $(K_n)_{n\in \nn}$ a basis at the identity in compact open subgroups with respect to which $G$ has $\eqref{eq:etoileH}$. Using repeatedly \cref{lem:monotonicity of phi/mu for crossed products} we obtain an index $n_0\in \nn$ and a sequence $(g_n)_{n\geq n_0}$ such that $g_{n_0}K_{n_0}\subset gK$ and for all $n\geq n_0$, 
        we have the inclusion $g_{n+1}K_{n+1}\subset g_{n}K_n$ and the inequality $\frac{\norm{\fc{z}{g_{n+1}K_{n+1}}}_{\psi}}{\mu_G(K_{n+1})}\geq \frac{\norm{\fc{z}{gK}}_{\psi}}{\mu_G(K)}$. The sequence $(g_n)_{n\geq n_0}$ converges to a non-trivial element that we denote by $h$. \Cref{lem:hilbert inequality for crossed products} gives 
        \[ \norm{z}_{\infty}^2
        \geq \ccc{H}{hK_n}\norm{\fc{z}{hK_n}}_{\psi}^2
        \geq \ccc{H}{hK_n}\mu_G(K_n)^2\frac{\norm{\fc{z}{gK}}_{\psi}^2}{\mu_G(K)^2}.
        \]
        Therefore, $\limsup_{n\in \nn}\ccc{H}{hK_n}\mu_G(K_n)^2=\infty$, and thus $\fc{z}{gK}=0$ which is a contradiction. By \Cref{lem:condition for nullity of an element in a crossed product}, $z\in M$.
    \end{proof}
    As an application of the previous theorem, we obtain factoriality for the group-measure space construction associated to the (non-free) action of $\Ndk$ on $\bTdk$ endowed with its natural visual measure. 
    \begin{Prop}\label{Prop:factoriality of the action of Neretin on its boundary}
    Let $\Ndk$ be a Neretin group, and $\Odk$ be its open subgroup of local isometries. Let $\nu$ be the visual probability measure on $\bTdk$ and $\alpha:\Ndk\acts \Linf(\bTdk,\nu)$ the action induced by the non-singular action of $\Ndk$ on $(\bTdk,\nu)$. Then 
    \[ \L(\Odk)'\cap (\Linf(\bTdk,\nu) \rtimes_{\alpha} \Ndk)=\cc.\] 
    Moreover, $\Linf(\bTdk,\nu) \rtimes_{\alpha} \Ndk$ is a type $\III_{\frac{1}{d}}$ factor.
\end{Prop}
\begin{proof}
    Put $N=\Linf(\bTdk,\nu) \rtimes_{\alpha} \Ndk$. 
    The irreducible inclusion $\L(\Odk)'\cap N=\cc$ follows directly from  \Cref{the:factoriality for crossed products} and \Cref{thm:Neretin} together with the fact that $\Odk\acts (\bTdk,\nu)$ is an ergodic probability measure preserving action. Let $\phi$ be the canonical weight on $\Linf(\bTdk,\nu)$ given by integration against $\nu$, and $\widetilde{\phi}$ be its dual weight on $N$. Then $\sigma_t^{\widetilde{\phi}}(u_s)=u_s\left(\frac{d\nu\circ \alpha_s}{d\nu}\right)^{it}$ for $s\in \N_{d,k}$ and $\sigma_t^{\widetilde{\phi}}(f)=f$ for $f\in \Linf(\bTdk,\nu)$. Note that $\L(\Odk)$ is contained in the centralizer $N_{\widetilde{\phi}}$ of $\widetilde{\phi}$. In particular, this implies that $N_{\widetilde{\phi}}$ itself is a factor. As observed on indicator functions of clopen subsets of $\bTdk$, $\left(\frac{d\nu\circ \alpha_s}{d\nu}\right)\in \Lun(\bTdk,\nu)$ takes its values in $\left(\frac{1}{d}\right)^{\zz}$, and therefore $\Tinv(N)=\left(\frac{2\pi}{\log(1/d)}\right)\zz$ by \cite[Theorème 1.3.2]{Connes1973Classification}, where $\Tinv(N)$ denotes the $\Tinv$-invariant of the von Neumann algebra $N$, defined in \cite[Definition 1.3.1]{Connes1973Classification}. Therefore, this factor must be of type $\III_0$ or $\III_{1/d}$ by \cite[Théorème 3.4.1]{Connes1973Classification}. However, \cite[Corollary 2.10]{ConnesTakesaki1977FlowWeights} precludes a type $\III$ factor with a faithful normal semifinite weight whose centralizer is factorial from being of type $\III_0$; therefore $N$ is a type $\III_{1/d}$ factor.
\end{proof}
    \section*{Acknowledgments} We warmly thank our supervisor Amine Marrakchi for his constant support during this project. We are grateful to Sasha Bontemps for stimulating discussions on HNN extensions, and to Adrien Le Boudec and Alejandra Garrido for their insightful comments on this work.
\printbibliography
\end{document}